\documentclass[aap,preprint]{imsart}

\RequirePackage{amsthm,amsmath,amsfonts,amssymb}
\RequirePackage[numbers]{natbib}

\startlocaldefs

\usepackage{dsfont}
\usepackage{comment}
\usepackage{textcomp}
\usepackage{amssymb}
\usepackage{bbm}
\usepackage{fancyhdr}
\usepackage{amsmath}
\usepackage{amsbsy,amsthm}
\usepackage{amscd}
\usepackage{latexsym}
\usepackage{graphicx}   
\usepackage{pdfsync}
\usepackage{blkarray}
\usepackage{multirow}
\usepackage{hyperref}
\usepackage{xcolor}
\usepackage{enumerate}

\usepackage{tcolorbox}

\newcommand{\R}{\mathbb{R}}

\newcommand{\cL}{{\cal L}}

\newtheorem{lemma}{Lemma}
\newtheorem{theorem}{Theorem}

\newtheorem{proposition}{Proposition}
\newtheorem{definition}{Definition}
\newtheorem{claim}{Claim}

\allowdisplaybreaks

\numberwithin{equation}{section}

\def\IND{\mathbbm{1}}

\newcommand{\ol}{\overline}
\newcommand{\wt}{\widetilde}
\newcommand{\wh}{\widehat}

\newcommand{\EXP}{\mathbb{E}}
\newcommand{\PROB}{\mathbb{P}}

\newcommand{\var}{\mathrm{Var}}
\newcommand{\cov}{\mathrm{Cov}}

\newcommand{\defeq}{\stackrel{\mathrm{def.}}{=}}

\endlocaldefs

\begin{document}

\begin{frontmatter}

\title{Broadcasting on random recursive trees\thanksref{T1}}
\runtitle{Broadcasting on random recursive trees}
\thankstext{T1}{G\'abor Lugosi was supported by the Spanish Ministry of Economy and Competitiveness, Grant MTM2015-67304-P and FEDER, EU; ``High-dimensional problems in structured probabilistic models - Ayudas Fundaci\'on BBVA a Equipos de Investigaci\'on Cientifica 2017''; and Google Focused Award ``Algorithms and Learning for AI''. Louigi Addario-Berry and Luc Devroye were supported by NSERC Discovery Grants and by an FRQNT Team Research Grant.}

\begin{aug}

\author[A]{\fnms{Louigi} \snm{Addario-Berry}\ead[label=e1]{louigi.addario@mcgill.ca}},
\author[B]{\fnms{Luc} \snm{Devroye}\ead[label=e2]{lucdevroye@gmail.com}},
\author[C]{\fnms{G\'abor} \snm{Lugosi}\ead[label=e3]{gabor.lugosi@upf.edu}},
\and
\author[D]{\fnms{Vasiliki} \snm{Velona}\ead[label=e4]{vasiliki.velona@upf.edu}}

\address[A]{Department of Mathematics and Statistics, McGill University, Montreal, Canada, \printead{e1}}
\address[B]{School of Computer Science, McGill University, Montreal, Canada, \printead{e2}}
\address[C]{Department of Economics and Business, Pompeu
  Fabra University, Barcelona, Spain; ICREA, Pg. Llu\'{\i}s Companys 23, 08010 Barcelona, Spain, and Barcelona Graduate School of Economics, \printead{e3}}
\address[D]{Department of Mathematics, Polytechnic University of Catalonia \& Department of Economics and Business, Pompeu
  Fabra University, Barcelona, Spain, \printead{e4}}
\end{aug}

\begin{abstract}
We study the broadcasting problem when the underlying tree is a random
recursive tree. The root of the tree has a random bit value assigned. 
Every other vertex has the same bit value as its parent with 
probability $1-q$ and the opposite value with probability $q$, where
$q \in [0,1]$.
The broadcasting problem consists in estimating the value of the root
bit upon observing the unlabeled tree, together
with the bit value associated with every vertex. 
In a more difficult version of the problem, the unlabeled tree is
observed  but only the bit values of the leaves are observed.
When the underlying tree is a uniform random recursive tree, 
in both variants of the problem we characterize the values of $q$
for which the optimal reconstruction method has a probability of 
error bounded away from $1/2$. We also show that the probability of
error is bounded by a constant times $q$. 
Two simple reconstruction rules are analyzed in detail. One of them is
the simple majority vote, the other is the bit value of the centroid
of the tree.
Most results are extended 
to linear preferential attachment trees as well.
\end{abstract}

\begin{keyword}
\kwd{broadcasting problem} 
\kwd{random trees}
\kwd{uniform attachment}
\kwd{preferential attachment}
\end{keyword}
\end{frontmatter}

\newpage
\tableofcontents

\section{Introduction}
\label{sec:intro}

\subsection{The broadcasting problem}

The \emph{broadcasting problem} on trees may be defined as follows. Let $T_n$ be a rooted tree on $n+1$
vertices. The vertices are labeled by $\{0,1,\ldots,n\}$ and the root has label $0$. The \emph{parent} $p_i$ of a vertex $i\in \{1,\ldots,n\}$
is the unique vertex on the path between the root and vertex $i$ that is connected to $i$ by an edge.
Each vertex is assigned a bit value
$B_i\in \{-1,1\}$ generated by the following random mechanism: the root bit obtains a bit uniformly at random,
while all other vertices have the same bit value as their parent with 
probability $1-q$ and the opposite value with probability $q$, where $q \in [0,1]$. In other words, for 
 $i\in \{1,\ldots,n\}$, 
\[
    B_i = B_{p_i}  Z_i
\]
where $Z_1,\ldots,Z_n$ are independent random variables taking values
in $\{-1,1\}$ with $\PROB\{Z_i=-1\}=q$. 

We consider the problem of estimating the value of the root bit $B_0$, upon observing the \emph{unlabeled} tree $T_n$, together
with the bit value associated with every vertex. (Note that since the vertex labels are not observed, the identity of the root is
not known.) We call this the \emph{root-bit reconstruction problem}.

In a more difficult version of the problem, the unlabeled tree is observed  but only the bit values of the \emph{leaves} are observed.
We refer to this variant as the problem of \emph{reconstruction from leaf bits}.

In this paper we consider these problems when the underlying tree is a \emph{random recursive tree}. Such trees are
grown, starting from the root vertex $0$, by adding vertices recursively one-by-one, according to some simple random rule. 
The simplest and most important
example is the \emph{uniform random recursive tree} in which, for each $i\in \{1,\ldots,n\}$, vertex $i$ attaches with an edge to
a vertex picked uniformly at random among vertices $0,1,\ldots,i-1$.

We also consider \emph{preferential attachment trees}. In such a tree vertex $i$ chooses a vertex among $0,1,\ldots,i-1$
such that the probability of attaching to vertex $j\in \{0,1,\ldots,i-1\}$ depends on the outdegree $D_j^+\left(i-1\right)$ of vertex $j$
at the time vertex $i$ is attached. (The outdegree of a vertex $j$ is the
number of vertices with index larger than $j$ attached to $j$.) We consider \emph{linear} 
preferential attachment models. In such a model,
\[
     \PROB\{i \sim j\} = \frac{D_j^+\left(i-1\right)+\beta}{\sum_{k=0}^{i-1}D_k^+\left(i-1\right)+\beta}~,
\]
where 
$\beta > 0$ is a parameter. 


The root-bit reconstruction problem is a binary classification problem, in which one observes an unlabeled tree $T_n$ generated by one of the random attachment mechanisms
defined above, together with the bit values assigned to all $n+1$ vertices. (In the problem of reconstruction from leaf bits,
only the bit values assigned to the leaves of $T_n$ are observed.) Based on this observation, one guesses the
value of the root bit $B_0$ by an estimate $\wh{b}$. The probability of error (or risk) is denoted by
\[
    R(n,q) = \PROB\left\{ \wh{b} \neq B_0 \right\}~.
\]
In this paper we study the optimal risk
\begin{equation}
    R^*(n,q)= \inf R(n,q)~,\label{riskdef}
\end{equation}
where the infimum is taken over all estimators $\wh{b}$. In particular, we are interested in 
\[
   R^*(q)= \limsup_{n\to \infty} R^*(n,q)~.
\]
Clearly, $R^*(n,q) \le 1/2$ for all $n$ and $q$ and a principal
question of interest is for what values of $q$ one has $R^*(q)<1/2$
and how $R^*(q)$ depends on $q$ in both problems and under the various random attachment models.

We assume, for simplicity, that the generating 
mechanism of the tree and the value $q$ are known to the statistician. 

For convenience of presentation, we focus the discussion on the
uniform random recursive tree. Preferential attachment models are
discussed in Section \ref{sec:pref}.

Before discussing root-bit estimators, we make an easy observation.

\begin{proposition}
In the root-bit reconstruction problem and the reconstruction problem from leaf bits on a uniform random recursive
tree,
$R^*(q)\ge q/2$. In particular, $R^*(1)=1/2$. Moreover, $R^*(0)=0$.
\end{proposition}

\begin{proof}
With probability $q$, the bit values of vertex $0$ and vertex $1$ are different. Since these two vertices are statistically
indistinguishable after their labels are removed, on this event, any classification rule has a probability of error $1/2$.
\end{proof}

We begin by noting that an optimal classification rule, achieving
error probability equal to the minimal risk~\eqref{riskdef}, may be
explicitly determined.
To describe such a classification rule with minimal probability of error, 
we first recall some facts established by Bubeck, Devroye, and Lugosi
\cite{bubeck2017finding}.

A recursive labeling of a rooted tree $T=T_n$ on $n+1$ vertices is a
labeling of the vertices of the tree with integers in
$\{0,1,\ldots,n\}$ such that every vertex has a distinct label, and 
the labels on every path starting from the origin are increasing.
(Thus, the root has label $0$.)

Write $V\left(T\right)$ for the set of vertices of a tree $T$. Given vertices $u,v\in V\left(T\right)$, we denote by $T^v_{u \downarrow}$ the subtree of $T$ that contains all
vertices whose path to $v$ includes $u$.

For a vertex $v\in V\left(T\right)$, we denote by $\text{Aut}\left(v,T\right)$ the number of vertices equivalent to $v$ under graph isomorphism. Formally,
$$\text{Aut}\left(v,T\right)=\left|\{w\in V\left(T\right):\exists\;\;\text{graph automorphism $\phi :T\rightarrow T$ such that $\phi \left(v\right)=w$}\}\right|$$
 Let $u_1,\dots ,u_j$ be the children of $v$ 
 and consider the subtrees $T_{u_1\downarrow}^0,\dots ,T_{u_j\downarrow}^0$. These subtrees belong to rooted graph isomorphism classes $S_1,\dots ,S_m$. For $i\in \left[m\right]$, let $\ell_i$ be the number of representatives of $S_i$, formally $\ell_i\defeq\left|\left\{k\in \left[j\right]: T_{u_k\downarrow}^0\in S_i\right\}\right|$. Moreover, let $\ol{\text{Aut}}\left(T_{v\downarrow}^0\right)\defeq\prod _{i=1}^m\ell_i!$. 

It is shown in \cite[Proposition 1]{bubeck2017finding} that, given a tree
$T$ on $n+1$ vertices, for any node $v\in T$, the number of recursive labelings of $T$ such that $u$ has label $0$ equals
\[
\frac{(n+1)!}{\prod_{v \in V(T) \setminus \cL(T)} \left(|T^u_{v
      \downarrow}| \cdot  \ol{\text{Aut}}\left(T_{v\downarrow}^u\right)\right)} ~,
\]
where $\cL(T)$ is the set of leaves of $T$. As a consequence, we have that, given an unlabeled tree $T$, 
the likelihood of each vertex $u$ being the root (under the uniform
attachment model) is proportional to the function
\begin{equation}
\label{eq:likelihood}
 \lambda(u)=  \frac{1}{\text{Aut}\left(u,T\right)\prod_{v\in V\setminus
   L\left(T,u\right)}\left(\left|T_{v\downarrow}^u\right|\cdot \ol{\text{Aut}}\left(T_{v\downarrow}^u\right)\right)}~.
\end{equation}
By the conditional independence of the generation of the uniform attachment tree
and the process of broadcasting the root bit, one easily obtains the
following.

\begin{proposition}
For the root-bit reconstruction problem on a uniform random recursive tree $T$, the following estimator $b^*$ of the root
bit $B_0$ minimizes the probability of error:
\[
   b^* = \left\{ \begin{array}{ll}
                       1 & \text{if} \quad \displaystyle{\sum_{u \in V(T): B_u=1}
                           \lambda(u) > \sum_{u \in V(T): B_u=0}
                           \lambda(u)}  \\
                    0 & \text{otherwise.}    \end{array} \right.
\]
In other words, $\PROB\{ b^* \neq B_0\} = R(n,q)$.
\end{proposition}

The analysis of the optimal rule described above seems difficult. 
Instead,
we analyze various other classification methods.

\subsection{Main results}

In this section we present our main findings for the uniform attachment
model. 
Some of the results are extended to the linear preferential attachment models in
Section \ref{sec:pref}.

One of the main results of the paper is that the trivial lower bound $R^*(q)\ge q/2$ above is tight, up to a constant factor.
This may be surprising since it is not even entirely obvious whether there exists any $q>0$ for which $R^*(q)<1/2$. 

\begin{tcolorbox}
\begin{theorem}
\label{thm:main1}
Consider the root-bit reconstruction problem in a uniform random recursive tree.
Then
\[
R^*(q)\le q
\]
 for all $q\in [0,1]$.
In the reconstruction problem from leaf bits,
\[
R^*(q)\le 13 q
\]
 for all $q\in [0,1]$.
\end{theorem}
\end{tcolorbox}
 


Our other main result is that for the uniform random recursive tree,  
we characterize the values of $q$ for which $R^*(q) < 1/2$.

\begin{tcolorbox}
\begin{theorem}
\label{thm:main2}
Consider the broadcasting problem in a uniform random recursive tree.
\begin{enumerate}
\item
In the root-bit reconstruction problem $R^*(q) < 1/2$ if and only if $q \in [0,1)$.
\item
In the reconstruction problem from leaf bits, $R^*(q) < 1/2$ if and only if $q \in [0,1/2) \cup (1/2,1)$.
\end{enumerate}
\end{theorem}
\end{tcolorbox}

Note that in the reconstruction problem from leaf bits, one obviously
has $R^*(1/2) = 1/2$. This follows from the fact that, when $q=1/2$,
the bit values on the vertices of the tree are independent unbiased
coin tosses. With probability tending to one, the root of the tree is not 
a leaf and therefore its bit value is not observed. In all other
cases (except when $q=1$), an asymptotic probability of error
strictly smaller than $1/2$ is achievable.

Perhaps the conceptually simplest method is the \emph{majority} rule
that simply counts the number of observed vertices with both bit
values and decides according to the majority. Denote by
$\wh{b}_\text{maj}$ the majority. (In case of a voting tie we may
arbitrarily define $\wh{b}_\text{maj}=0$.) This
simple method has surprisingly good properties. Indeed, we prove the following
bound.

\begin{tcolorbox}
\begin{theorem}
\label{thm:majority}
Consider the broadcasting problem in a uniform random recursive tree.
Denote the probability of error of the majority vote by
\[
R^{\text{maj}} (n,q) =
\PROB\left\{ \wh{b}_\text{maj}\neq B_0 \right\}~.
\]
For both the root-bit reconstruction problem and the reconstruction problem from leaf bits, 
the following hold.
\begin{enumerate}
\item
There exists $c>0$ such that 
\[
\limsup_{n\to \infty} R^{\text{maj}} (n,q) \le cq \quad \text{for all} \ q \in [0,1]~.
\]

\item
\[
\limsup_{n\to\infty} R^{ \text{maj}} (n,q) < 1/2 \quad \text{if } \ q \in [0,1/4)
\]
 and 
\[
\limsup_{n\to\infty} R^{\text{maj}} (n,q) = 1/2 \quad \text{if } \ q \in [1/4,1/2]~.
\]
\end{enumerate}
\end{theorem}
\end{tcolorbox}

A quite different approach is based on the idea that, if one is able
to identify a vertex that is close to the root, then the bit value
associated to that vertex is correlated to that of the root bit,
giving rise to a meaningful guess of the root bit. The possibilities
and limitations of identifying the root vertex have been thoroughly
studied in recent years--see Section \ref{sec:relatedwork} for
references. 

A simple and natural candidate for an estimate of the root 
is the \emph{centroid} of the tree. In order to define the centroid of a tree $T$,
we need some notation.
The   \emph{neighborhood} of a vertex $v$, that is, the set of vertices in $T$ connected to $v$, is denoted by $N(v)$. 

Define $\phi :V(T)\rightarrow \mathbb{R}^+$  by
\[
\phi (v)=\max _{u\in N(v)}\left| V\left(T^v_{u\downarrow}\right)\right|
\]
and define a \emph{centroid} of $T$ by 
\[
v^*=\arg\min _{v\in V(T)}\phi (v)~.
\]
It is well known that a tree can have at most two centroids. In fact,
$\phi  (v^*)\leq \frac{|V(T)|}{2}$ and there are at most two vertices that attain the minimum value. If there are two of them, then they are connected with an edge (Harary \cite{harary6graph}).

Equipped with this notion, now we may define an estimator $\wh{b}_\text{cent}$ of the root bit in a natural way:
(1)  in the root-bit reconstruction problem, $\wh{b}_\text{cent}=B_{v^*}$ is the bit value of an arbitrary centroid $v^*$ of $T$;
(2) in the reconstruction problem from leaf bits, let $v^*$ be a centroid of $T$, let $v^\circ$ be a leaf closest to $v^*$, and let
$\wh{b}_\text{cent}=B_{v^\circ}$ be the associated bit value.

We call this estimator the \emph{centroid rule}.

\begin{tcolorbox}
\begin{theorem}
\label{thm:centroid}
Consider the broadcasting problem in a uniform random recursive tree.
Denote the probability of error of the centroid rule by
\[
R^{\text{cent}} (n,q) =
\PROB\left\{ \wh{b}_\text{cent}\neq B_0 \right\}~.
\]
For the root-bit reconstruction problem,
\[
\limsup_{n\to \infty} R^{\text{cent}} (n,q) \le  q \quad \text{for all} \ q \in [0,1]
\]
and
\[
\limsup_{n\to\infty} R^{\text{cent}} (n,q) \leq
  \frac{\log 2}{2} \approx 0.34 \quad \text{for all} \ q \le 1/2~.
\]
For the reconstruction problem from leaf bits, 
\[
\limsup_{n\to \infty} R^{\text{cent}} (n,q) \le  13 q \quad \text{for all} \ q \in [0,1]~.
\]
Moreover,
\[
\limsup_{n\to\infty} R^{\text{cent}} (n,q) < 1/2 \quad \text{for all} \ q < 1/2~.
\]
\end{theorem}
\end{tcolorbox}

Clearly, Theorem \ref{thm:centroid} implies Theorem \ref{thm:main1}. In order to prove Theorem \ref{thm:main2}, 
we need to construct an estimator of the root bit that performs better than random guessing when $q\in (1/2,1)$.
This construction is described in Section \ref{sec:greaterthanhalf}, together with the proof that its asymptotic probability of error
is better than $1/2$.

The rest of the paper is organized as follows. In Section \ref{sec:majority} we analyze the majority rule and prove 
Theorem \ref{thm:majority}.
In Section \ref{sec:centroid} the analysis of the centroid rule is presented and Theorem \ref{thm:centroid} is proved.
In Section \ref{sec:greaterthanhalf} we complete the proof of  Theorem \ref{thm:main2}.

Finally, in Section \ref{sec:pref} the main results are extended to linear preferential attachment trees.

\subsection{Related work}
\label{sec:relatedwork}

The broadcasting problem on trees has a long and rich history. The form studied here was 
proposed by Evans, Kenyon, Peres, and Schulman \cite{evans2000broadcasting}. We refer to
this paper for the background of the problem and related literature.
In the broadcasting problem of \cite{evans2000broadcasting}, a bit is transmitted from each 
node to its children recursively, beginning from the root vertex. Each time the bit is transmitted between two nodes, the value of the bit 
is flipped with some probability. The authors study the problem of reconstructing the bit value of the root, 
based on the bit values of all vertices at distance $k$ from the root. 
They establish a sharp threshold for the probability of reconstruction as $k$ goes to infinity, depending on the tree's \emph{branching number}. 
Variants of this problem for asymmetric flip probabilities, non-binary vertex values, and perturbations have been studied by Sly~\cite{sly2011}, Mossel \cite{mossel2001reconstruction}, Janson and Mossel \cite{janson2004}.
A sample of recent progress and related results includes
Jain, Koehler, Liu and Mossel ~\cite{pmlr-v99-jain19b}
Mossel~\cite{mossel2004phase}, 
Daskalakis, Mossel, and Roch \cite{daskalakis2006optimal,daskalakis2011evolutionary}. 
M{\'e}zard and Montanari~\cite{Mezard2006},
 Mossel, Neeman, and Sly~\cite{pmlr-v35-mossel14},
Moitra, Mossel, and Sandon \cite{moitra2019circuit}, and
Makur, Mossel, and Polyanskiy \cite{makur2019broadcasting}.

As far as we know, the broadcasting problem has not been studied for random recursive trees.
In the vast majority of the literature on the broadcasting problem, the location of the root is assumed to be known. 
Of course, in this case the reconstruction problem is meaningful only if the bit values near the root are not observed.
The types of trees that are generally considered are such that, even if the root is not identified, it is easy to locate. 
In the problems that we consider, the trees are random recursive trees where localizing the root is a nontrivial issue.
Hence, both the root-bit reconstruction problem and the  problem of reconstruction from leaf bits are meaningful.
The structure of the tree plays an important role in the solution of both problems.

The problem of localizing the root in different models of random recursive trees has been studied by
Haigh \cite{haigh1970recovery},
Shah and Zaman \cite{SZ11},
Bubeck, Devroye, and Lugosi~\cite{bubeck2017finding}. 
For diverse results on closely related problems, see
Curien, Duquesne, Kortchemski, and Manolescu \cite{curien2014scaling},
Bubeck, Mossel, and  R{\'a}cz \cite{bubeck2015influence}, 
Bubeck, Eldan, Mossel, and  R{\'a}cz  \cite{bubeck2017trees}, 
Khim and Loh \cite{khim2016confidence}, 
Jog and Loh \cite{jog2016analysis},
Lugosi and Pereira \cite{lugosi2019finding},  and
Devroye and Reddad \cite{reddad2019discovery}. 


\section{Majority rule -- proof of Theorem \ref{thm:majority}}
\label{sec:majority}

In this section we analyze the majority rule and prove Theorem
\ref{thm:majority}. First we consider the root-bit reconstruction
problem, that is, we assume that the bit values are observed at every
vertex of the tree. In this case $\wh{b}^{\text{maj}}$ denotes the
majority vote among all bit values.
 In Section \ref{sec:leaves} we extend the argument
for the reconstruction problem from leaf bits.

Observe that the number of vertices in the uniform random recursive tree $T_n$ with bit value $B_0$ 
is distributed as the number of black balls in a P\'olya urn of black
and white balls with
random replacements defined as follows: initially, there is one black
ball in the urn. For $i=1,2,\ldots$, at time $i$, a uniformly random ball is
selected from the urn. The ball is returned to the urn
together with 
a new ball whose color is decided according to a Bernoulli$(q)$ coin
toss. If the value is $1$ (which happens with probability $q$), 
the color of the new ball is the opposite of the selected
one. Otherwise the new ball has the same color
as that of the selected ball. 

Such randomized urn processes have been thoroughly studied. 
In particular, early results can be traced back to Wei~\cite{wei1979generalized} and depend on results by Athreya and Karlin~\cite{athreya1967limit} concerning random multi-type trees.
More recently, Janson~\cite{janson2004functional} and 
Knape and Neininger \cite{KnNe14} proved general limit laws that may be used to analyze 
the probability of error of the majority rule. 

Instead of using these limit laws, our starting point is a 
decomposition of the uniform random recursive tree defined below. 
This methodology allows us to prove the first inequality of Theorem
\ref{thm:majority} in an elementary way. Moreover, this decomposition 
may be used to treat the case of the reconstruction problem from leaf
bits in a straightforward fashion. The same technique will also prove
useful in analyzing the majority vote in the preferential attachment
models. 

In Sections \ref{sec:123} and \ref{sec:q=1/4} we use Janson's limit
theorems to derive qualitative results on the probability of error of
the majority rule.

In this entire section we assume that $q\le 1/2$. The conclusions of the theorem hold trivially for $q\geq \frac{1}{2}$.


\subsection{A decomposition of the URRT}

It is convenient to decompose the uniform random recursive tree (URRT)
as follows. First, the URRT is generated in the standard way, without 
attached bit values. Then 
we identify all nodes apart from the root as follows:
\begin{itemize}
\item with probability $2q$, they are \emph{marked}. Then there is a coin flip $\xi$ that takes values uniformly at random in  $\{-1,1\}$ and determines if a marked node takes the same bit value as its parent or it flips.
\item with probability $1-2q$ they are \emph{not marked}. These nodes do not perform a flip, and thus have the same bit value as their parent.
\end{itemize}
\begin{figure}
\centering
\includegraphics[scale=.6]{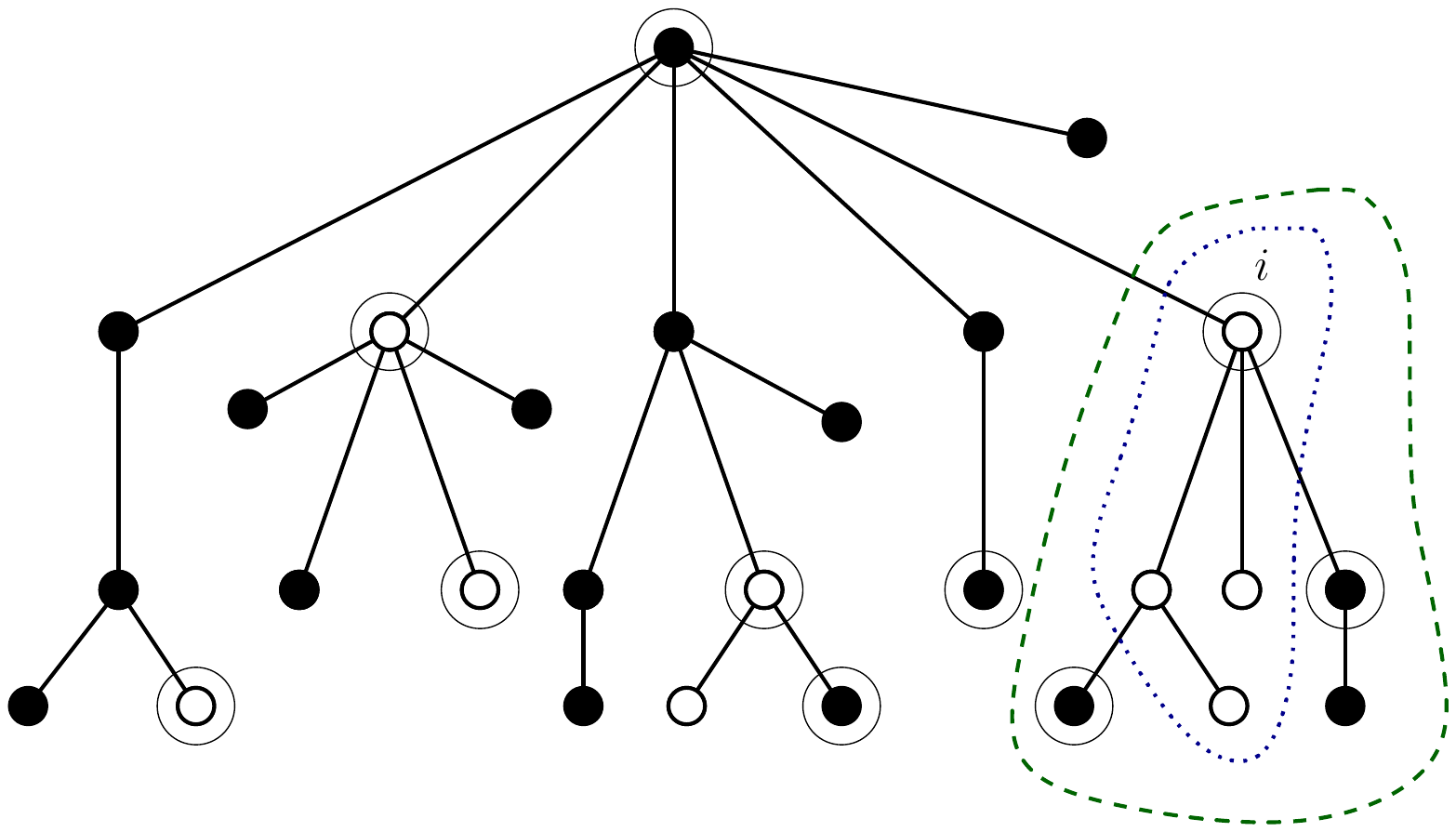}
\label{picture2}
\caption{Illustration of the decomposition of a tree. The vertices enclosed by a circle are marked. The subtree that is enclosed by a dotted curve is $\widetilde{T}_i$. The subtree that is enclosed by a dashed curve is $T^0_{i\downarrow}$. }
\end{figure}
The root and marked nodes become roots of subtrees that are disjoint and shatter the uniform recursive tree into many pieces. Each of the subtrees consists of nodes of the same bit value necessarily, and the roots have the bit value of their original parent if $\xi =1$ and different otherwise (if $\xi =-1$). We recall that nodes are numbered $0$ through $n$, where $0$ is the root. The node variables are, for node $i$:
\begin{itemize}
\item $p_i\in \{0,\dots , i-1\}$: the uniform random index of its parent
\item $m_i\in \{0,1\}$: a Bernoulli$(2q)$ random variable: 1 indicates marking
\item $\xi _i\in \{-1,1\}$: a Rademacher random variable used for flipping bit values: $\mathbb{P}\left[\xi _i =1\right]=\frac{1}{2}.$
\end{itemize}
Note that, for each $i\in \{1,\ldots,n\}$, $p_i,m_i$, and $\xi_i$ are
independent. Moreover, the sequence $\left(\left(p_i,m_i,\xi _i\right),1\leq i\leq n\right)$ is independent. Let $B_i$ be the bit value in $\{-1,1\}$ of node $i$, with $B_0=1$. We set
\[ B_i=
\left\{
\begin{array}{ll}
 B_{p_i}, & \mathrm{if}\;\;m_i=0\;\;\mathrm{(no}\;\;\mathrm{marking)}\;\;\mathrm{or}\;\;\mathrm{if}\;\;m_i=1,\xi _i=+1\;\;\mathrm{(no}\;\;\mathrm{flipping)} \\
-B_{p_i}, & \mathrm{if}\;\;m_i=1,\xi _i=-1\\ 
       \end{array} 
\right. 
\]
Formally, $B_i=\left(m_i\xi _i+\left(1-m_i\right)\right)B_{p_i}$. Note that 
\begin{itemize}\item The shape of the URRT depends only upon $p_1,\dots p_n$. 
\item The decomposition of the tree into subtrees depends upon $p_1,\dots p_n$ and $m_1,\dots m_n$.
\item The bit counting algorithm (that outputs the majority) uses $\xi _1,\dots \xi _n$ as well as the two other sequences.
\end{itemize}
Let $\wt{T}_i$ be the maximal size subtree of $T_{i\downarrow}^0$ with root $i$ and
homogeneous bit values, such that all its vertices
apart from $i$ are unmarked ($i$ can be either marked or unmarked). See Figure~\ref{picture2} for an illustration. We
write $N_i=|\wt{T}_i|.$ 

\subsection{Linear upper bound for the probability of error}
\label{sec:linbound}

Here we prove that there exists a universal constant $c$ such that
\begin{equation}
\label{eq:majlin}
\limsup_{n\to \infty} R^{\text{maj}} (n,q) \le cq \quad \text{for all} \ q \in [0,1]~.
\end{equation}
Taking $c\ge 8$, we may assume that $q\le 1/8$.

The difference between the number of nodes of value $1$ and those of value $-1$ is given by 
$$\Delta \overset{\mathrm{def}}{=}N_0+\sum _{i=1}^nN_iB_{p_i}\xi
_im_i.$$ In this formula, we only count subtrees corresponding to
vertices with $m_i=1$, and add the vertex count ($N_i$) to the
$B_{p_i}\xi_i$ side. As the $\xi _i$'s are independent of the rest of
the variables, we have 
\begin{equation}
\EXP \left[\Delta\right]=\EXP \left[N_0\right]~.
\label{exp1}\end{equation}
Also, by first conditioning on everything but the $\xi _i$'s, we have 
\[
\EXP \left[\Delta ^2\right] =\EXP \left[N_0^2\right]+\sum
                                          _{i=1}^n\EXP
                                          \left[N_i^2B^2_{p_i}m_i\right]=\EXP
                                          \left[N_0^2\right]+2q\sum
                                          _{i=1}^n\EXP
                                          \left[N_i^2\right]~.
\]
So, 
\[
\var\left[\Delta \right]=\var\left[N_0 \right]+2q\sum _{i=1}^n\EXP \left[N_i^2\right]~.
\]
By Chebyshev's inequality,
\begin{eqnarray*}
\PROB\left\{ \wh{b}_\text{maj}\neq B_0 \right\}
&\leq & 
\PROB \left\{\Delta \leq 0\right\}
\leq \frac{\var\left[\Delta\right]}{\left(\EXP
        \left[\Delta\right]\right)^2}
\\
& = & \frac{\var\left[N_0\right]}{\left(\EXP
      \left[N_0\right]\right)^2}+2q\frac{\sum _{i=1}^n\EXP
      \left[N_i^2\right]}{\left(\EXP \left[N_0\right]\right)^2}~.
\end{eqnarray*}
In Lemmas \ref{lem:10}, \ref{lem:12}, and \ref{lem:14}, stated and proved in Section~\ref{sec:study}, we
establish bounds for the first and second moments of $N_i$. These
bounds imply (\ref{eq:majlin}) as follows.

Let $\zeta(\alpha)= \sum_{i=1}^\infty 1/i^{\alpha}$ be the Riemann
zeta function and let $\wt\zeta(\alpha)= \sum_{i=1}^\infty (\log
i)/i^{\alpha}$. Note that both functions are finite and decreasing for $\alpha>1$.
By Lemmas  \ref{lem:10} and \ref{lem:14},
\begin{eqnarray*}
\lefteqn{
\frac{\var\left[N_0\right]}{\left(\EXP
      \left[N_0\right]\right)^2}   } \\
&\leq & 2qe^4(4+e) \zeta(2-4q) + 2qe^4n^{-(1-4q)}+
12e^5q^2  \wt\zeta(2-4q) + 4e^4q^2 n^{-(1-4q)}\log n\\
& \le &
c_1 q +
c_2 q^2 + o_n(1)  \\
\end{eqnarray*}
with $c_1=2e^4(4+e)  \zeta(3/2)$ and  $c_2= 12e^5 
\wt\zeta(3/2)$, where we used the fact that $\zeta$ and $\wt\zeta$ are
decreasing functions and that $q\le 1/8$.


On the other hand, by Lemmas \ref{lem:10} and \ref{lem:12}, 
\[
\frac{\sum _{i=1}^n\EXP
      \left[N_i^2\right]}{\left(\EXP \left[N_0\right]\right)^2}
\le   e^4(4+e)  \zeta(2-4q) + n^{-(1-4q)}e^3
\leq \frac{c_1}{2} + o_n(1)~.
\]
Hence, for all $q\le 1/8$,
\[
 \PROB\left\{ \wh{b}_\text{maj}\neq B_0 \right\} \le 2c_1 q + c_2 q^2
 + o_n(1)~,
\]
proving (\ref{eq:majlin}).

\subsection{Majority is better than random guessing for $q<1/4$}
\label{sec:123}

Next we show that
\begin{equation}
\label{eq:maj}
\limsup_{n\to \infty} R^{\text{maj}} (n,q) < \frac{1}{2} \quad
\text{for all} \ q < \frac{1}{4}~.
\end{equation}

To this end, we may apply Janson's~\cite{janson2004functional} limit
theorems for P\'olya urns with randomized replacements. 

Consider first the model when bit values are observed at every vertex
of the tree. 
Recall from the introduction of this section that the number of
vertices with bit value $B_0$ may be represented by the number of
white balls in a P\'olya urn of white and black balls, initialized with one white ball. At each
time, a random ball is drawn. The drawn ball is returned to the urn,
together with another ball whose color is the same as the drawn one
with probability $1-q$ and has opposite color with probability $q$.
The asymptotic 
distribution of the balls is determined by the eigenvalues and eigenvectors of the 
transpose of the matrix of the expected number of returned balls. 
In this case, the matrix is simply
\begin{equation*}
\left( {\begin{array}{cc}
   1-q & q  \\
   q & 1-q  \\    
 \end{array} } \right)~,
\end{equation*} 
whose eigenvalues are $1$ and $1-2q$. 
If $q<1/4$, by ~\cite[Theorem 3.24]{janson2004functional},
 \[
    \frac{\Delta - \EXP \Delta}{n^{1-2q}}
\]
converges, in distribution, to a random variable whose distribution is
symmetric about zero and has a positive density at $0$. Since 
\[
 \frac{\EXP \Delta}{n^{1-2q}} \geq \frac{1}{e\Gamma(2-2q)}
\]
by (\ref{exp1}) and the calculations in Lemmas~\ref{lem:8} and \ref{lem:10} below,
it follows that
\begin{eqnarray*}
\limsup_{n\to \infty} \PROB\left\{ \wh{b}_\text{maj}\neq B_0 \right\}
&\leq & 
\limsup_{n\to \infty} \PROB \left\{\Delta \leq 0\right\} \\
& = &
\limsup_{n\to \infty}  \PROB \left\{\frac{\Delta - \EXP
      \Delta}{n^{1-2q}} \leq   - \frac{\EXP \Delta}{n^{1-2q}}\right\}  \\
& < & \frac{1}{2}~,
\end{eqnarray*}
proving (\ref{eq:maj}).

The majority rule in the leaf-bit reconstruction problem may also be
studied using P\'olya urns with random replacements. In this case the
urn has four colors, corresponding to (1)  leaf vertices whose bit value
equals $B_0$; (2)  leaf vertices whose bit value
equals $1-B_0$; (3)  internal vertices whose bit value
equals $B_0$; (4) internal vertices whose bit value
equals $1-B_0$. 

Initially, there is one ball of type (1) and no balls of any other
type in the urn. When a ball of type (1) is drawn, it is replaced by a
ball of type (3).  With probability $1-q$, an additional ball of
type (1) is added to the urn, and with probability $q$ a ball of type
(2) is added, etc. The resulting replacement matrix is
\begin{equation*}
\left( {\begin{array}{cccc}
   -q & q & 1 & 0  \\
    q & -q & 0 & 1  \\   
     1- q & q & 0 & 0 \\  
         q & 1-q & 0 & 0 \\  
 \end{array} } \right)
\end{equation*}
The eigenvalues of the transpose of this matrix are $1,1-2q,-1,-1$,
and once again \cite[Theorem 3.24]{janson2004functional} applies. Reasoning as previously and using Lemma~\ref{lem:leafmoments}, we have
that for $q<1/4$,
\begin{eqnarray*}
\limsup_{n\to \infty} \PROB\left\{ \wh{b}_\text{maj}\neq B_0 \right\}
& < & \frac{1}{2}~.
\end{eqnarray*}




\subsection{Majority is not better than random guessing for $q>1/4$}

Here we prove that
\begin{equation}
\label{eq:majlower}
\limsup_{n\to \infty} R^{\text{maj}} (n,q) = \frac{1}{2} \quad
\text{for all} \ q \in  (1/4,1/2]~.
\end{equation}

This follows easily from the decomposition of the URRT introduced
above and the following lemma:

\begin{lemma}[Rogozin, 1961~\cite{rogozin1961estimate}]
\label{rogozin} 
Let $\xi _1,\dots ,\xi_n$ be i.i.d.\ Bernoulli$\left(\frac{1}{2}\right)$ random variables. Then for any $\alpha_1,\dots ,\alpha _n$, all nonzero,
$$\sup_x \;\;\PROB \left\{\sum _{i=1}^n\xi_i\alpha _i=x\right\} \leq\frac{\gamma}{\sqrt{n}}$$
for some universal constant $\gamma$, uniformly over all choices of $\alpha _1,\dots \alpha _n$.
\end{lemma}
\noindent Indeed, 
\begin{eqnarray*}
\PROB\left\{ \wh{b}_\text{maj}\neq B_0 \right\}
&\geq & \PROB \left\{\Delta <0\right\}=\PROB \left\{\sum
        _{i=1}^nN_iB_{p_i}m_i\xi _i<-N_0\right\}
\\
&= & \frac{1}{2}\PROB \left\{\left|\sum _{i=1}^nN_iB_{p_i}m_i\xi
     _i\right|>N_0\right\}
\quad \text{(by symmetry)}\\
&\geq & \frac{1}{2}\EXP \left[\left(1-\frac{2\gamma (N_0+1)}{\sqrt{\sum _{i=1}^nm_i}}\right)_+\right]~.
\end{eqnarray*}
The inequality above follows 
by first conditioning on all but the $\xi _i$'s and using
Lemma~\ref{rogozin}. The latter expression is further lower bounded by 
\begin{align*}
&\frac{1}{2}\left(\EXP \left[\left(1-\frac{2\gamma (N_0+1)}{\sqrt{qn}}\right)_+\right]-\PROB \left\{\sum _{i=1}^nm_i<qn\right\}\right)\\
&\geq \frac{1}{2}\left(1-\frac{2\gamma \EXP
  \left[N_0+1\right]}{\sqrt{qn}}\right)_+
-\PROB \left\{\text{Binomial}(n,2q)<qn\right\} \tag*{(by Jensen's inequality)}\\
&= \frac{1}{2}-o_n\left(1\right)~,
\end{align*}
since $\EXP \left[N_0\right]=o\left(\sqrt{n}\right)$ when
$q>\frac{1}{4}$ by Lemma~\ref{lem:10}.

\subsection{Majority is not better than random guessing for $q=1/4$}
\label{sec:q=1/4}

In the ``critical'' case $q=1/4$, we may, once again, use the P\'olya
urn representation and the limit theorems of Janson \cite{janson2004functional}.
Indeed, by working as in Section~\ref{sec:123}, \cite[Theorem
3.23]{janson2004functional} applies and it implies  that
\[
    \frac{\Delta - \EXP \Delta}{n^{1/2}\log n}
\]
converges, in distribution, to a normal random variable. Since
\[
 \frac{\EXP \Delta}{n^{1/2}\log n} =o\left(1\right)
\]
by Lemmas~\ref{lem:8} and ~\ref{lem:10}, we have
\begin{eqnarray*}
\limsup_{n\to \infty} \PROB\left\{ \wh{b}_\text{maj}\neq B_0 \right\}
&\geq & 
\limsup_{n\to \infty} \PROB \left\{\Delta < 0\right\} \\
& = &
\limsup_{n\to \infty}  \PROB \left\{\frac{\Delta - \EXP
      \Delta}{n^{1/2}\log n} \leq   - \frac{\EXP \Delta}{n^{1/2}\log
      n}\right\}  =  \frac{1}{2}~.
\end{eqnarray*}

A similar computation may be performed for the case when only leaf-bits are observed.

\subsection{The study of $N_i$} \label{sec:study}

In this section we present the technical results used in the proofs of
this section. In particular, we bound the first and second moments of 
the random variables $N_i$ defined in the decomposition of the URRT.
We begin with two technical lemmas.

\begin{lemma}
\label{lem:8}
For all $i\geq 0$ and constant $\alpha\geq 0$, 
\[
\prod _{t=i}^{n-1}\left(1+\frac{\alpha}{t+1}\right)=\frac{\Gamma
  \left(\alpha +n+1\right)}{\Gamma \left(n+1\right)}\cdot\frac{\Gamma
  \left(i+1\right)}{\Gamma \left(\alpha +i+1\right)}~.
\]
\end{lemma}
\begin{proof}
\begin{equation}\label{eq:30}
\prod _{t=i}^{n-1}\left(1+\frac{\alpha}{t+1}\right)=\frac{\prod _{t=0}^{n-1}\left(\frac{\alpha+1+t}{1+t}\right)}{\prod _{t=0}^{i-1}\left(\frac{\alpha+1+t}{1+t}\right)}~. 
\end{equation}
Also,
\begin{equation*}
\prod _{t=0}^{n-1}\left(\frac{\alpha+1+t}{1+t}\right)=\frac{\Gamma \left(\alpha +n+1\right)}{\Gamma \left(\alpha +1\right)\Gamma \left(n+1\right)}~,
\end{equation*}
implying that~\eqref{eq:30} equals
$$\frac{\Gamma \left(\alpha +n+1\right)}{\Gamma \left(\alpha +1 \right)\Gamma \left(n+1\right)}\cdot\frac{\Gamma \left(\alpha +1\right)\Gamma \left(i+1\right)}{\Gamma \left(\alpha +i+1\right)}=\frac{\Gamma \left(\alpha +n+1\right)}{\Gamma \left(n+1\right)}\cdot\frac{\Gamma \left(i+1\right)}{\Gamma \left(\alpha +i+1\right)}~.$$
\end{proof}

\begin{lemma}\label{lem:9}
For $n\geq 1$ and $\alpha \in [0,1]$,
$$\left(\frac{n+1}{e}\right)^\alpha \leq \frac{\Gamma\left(\alpha +n+1\right) }{\Gamma \left (n+1\right)}\leq \left(n+1\right)^\alpha .$$
\end{lemma}
\begin{proof}
If $\text{Gamma} \left( n+1\right)$ denotes a Gamma random variable
   with parameters $(n+1,1)$, then
\begin{eqnarray*}
\frac{\Gamma \left(\alpha +n+1\right)}{\Gamma \left(n+1\right)} 
& = & \frac{\int _0^\infty x^{\alpha+n} e^{-x}dx}{\int _0^\infty
      x^ne^{-x}dx}   \\
& = &\EXP \left[\text{Gamma} \left( n+1\right)^\alpha\right]\\
 & \leq & \left(\EXP \left[\text{Gamma} \left(n+1\right)\right]\right)^\alpha =\left(n+1\right)^\alpha~,
\end{eqnarray*}
by Jensen's inequality. 
We show the lower bound by induction to $n$. For $n=1$ it holds for all $\alpha\in \left[0,1\right]$, since $\left(\frac{2}{e}\right)^\alpha\leq 1\leq\Gamma \left(2+\alpha\right).$ For larger $n$, note:
$$\frac{\Gamma\left(\alpha +n+1\right)}{\Gamma\left(n+1\right)}=\frac{n+\alpha}{n}\cdot\frac{\Gamma\left(\alpha +n\right)}{\Gamma\left(n\right)}\geq  \frac{n+\alpha}{n}\left(\frac{n}{e}\right)^\alpha\geq\left(\frac{n+1}{e}\right)^\alpha ~,$$
where the first inequality follows by induction hypothesis and the second since $\frac{n+\alpha}{n}\geq \left(\frac{n+1}{n}\right)^\alpha$.
\end{proof}
\begin{lemma}
\label{lem:10}
For all $i\geq 0$ and $q\leq \frac{1}{2}$,
\[
e^{-1}\left(\frac{n+1}{i+1}\right)^{1-2q}\leq \EXP \left[N_i\right]\leq e\left(\frac{n+1}{i+1}\right)^{1-2q}~.
\]
\end{lemma}

\begin{proof}
The statement follows immediately by Lemmas~\ref{lem:8} and
\ref{lem:9} by noting that 
\begin{equation}
\EXP \left[N_i\right]=\prod _{t=i}^{n-1}\left(1+\frac{1-2q}{t+1}\right)~.\label{eq:432}
\end{equation}
To see that~\eqref{eq:432} holds, define $Y_i=1$ and, for $t\in \{i,\ldots,n-1\}$, let  
\[
Y_{t+1}=Y_t+\beta_{1-2q}\beta_{Y_t/\left(t+1\right)}~.
\]
where each appearance of $\beta_x$ denotes an independent Bernoulli$(x)$ random variable. 
Clearly, $Y_t$ is distributed as the number of vertices counted by $N_i$ and which have label at most $t$. Hence $N_i$ has the same distribution as $Y_n$. For all $t\geq 1$, by conditioning on $Y_t$ we see that 
$$\EXP\left[Y_{t+1}\right]=\EXP\left[Y_{t}\right]\left(1+\frac{1-2q}{t+1}\right)~,$$ from which~\eqref{eq:432} is immediate. 
\end{proof}


\begin{lemma}\label{lem:12}
For all $i\geq 0$ and $q \leq \frac{1}{2}$,
\begin{align*}
\EXP \left[N_i^2\right]\leq\left(\frac{n+1}{i+1}\right)^{2-4q}e^{2(1-2q)}\left(4+e\right)+e\left(1-2q\right)~.
\end{align*}
\end{lemma}
\begin{proof}
We use the representation of $N_i$ introduced in the proof of Lemma \ref{lem:10}.
Consider the recurrence
\[
x_i=1,\;\;x_{t+1}=x_t\left(1+\frac{2\alpha}{t+1}\right)+f(t)~, \qquad i\leq
t\leq n~.
\]
In particular, we are interested in the case $\alpha=1-2q$,
$f(t)=\left(1-2q\right)\frac{\EXP \left[Y_t\right]}{t+1}$, and 
$x_t=\EXP \left[Y_t^2\right]$. The solution is given by 
\[
x_n=x_i\prod _{t=i}^{n-1}\left(1+\frac{2\alpha}{t+1}\right)
+\sum_{s=i+1}^{n-1}\prod_{t=s}^{n-1}\left(1+\frac{2\alpha}{t+1}\right)f\left(s-1\right)+f\left(n-1\right)~.
\]
Using Lemmas~\ref{lem:8},\ref{lem:9},\ref{lem:10} and the bound $f(t)\leq
\alpha \left(\frac{t+1}{i+1}\right)^\alpha \frac{e}{t+1}$, 
\begin{eqnarray*}
x_n &\leq & x_i\left(\frac{n+1}{i+1}\right)^{2\alpha} e^{2\alpha}+\sum _{s=i+1}^{n-1}\left(\frac{n+1}{s+1}\right)^{2\alpha}e^{2\alpha +1}\alpha \left(\frac{s}{i+1}\right)^\alpha \frac{1}{s}+\alpha e\\
 &= & \left(\frac{n+1}{i+1}\right)^{2\alpha}e^{2\alpha}\left(1+\sum _{s=i+1}^{n-1}\frac{s^\alpha\cdot e\alpha \left(i+1\right)^\alpha}{s\left(s+1\right)^{2\alpha}}\right)+\alpha e\\
  &\leq &\left(\frac{n+1}{i+1}\right)^{2\alpha}e^{2\alpha}\left(1+\frac{e\alpha }{i+1}+\sum _{s=i+2}^{n-1}\frac{e\alpha \left(i+1\right)^\alpha}{s^{1+\alpha}}\right)+\alpha e\\
  &\leq & \left(\frac{n+1}{i+1}\right)^{2\alpha}e^{2\alpha}\left(4+e\alpha \left(i+1\right)^\alpha\int ^\infty _{i+1}\frac{1}{s^{1+\alpha}}ds\right)+\alpha e\\
  &= & \left(\frac{n+1}{i+1}\right)^{2\alpha}e^{2\alpha}\left(4+\frac{e\alpha \left(i+1\right)^\alpha}{\alpha \left(i+1\right)^{\alpha}}\right)+\alpha e\\
    &\leq & \left(\frac{n+1}{i+1}\right)^{2\alpha}e^{2\alpha}\left(4+e\right)+\alpha e~.
\end{eqnarray*}
Replacing $\alpha$ by $1-2q$, we have 
\[
\EXP \left[N_i^2\right]\leq\left(\frac{n+1}{i+1}\right)^{2-4q}e^{2(1-2q)}\left(4+e\right)+e\left(1-2q\right)~.
\]
\end{proof}

Recall the notation $\zeta(\alpha)= \sum_{i=1}^\infty 1/i^{\alpha}$ and  $\wt\zeta(\alpha)= \sum_{i=1}^\infty (\log
i)/i^{\alpha}$.

\begin{lemma}
\label{lem:14}$\var(N_0)$ is bounded by 
\[  
2qe^2(4+e) (n+1)^{2-4q} \zeta(2-4q) + 2nqe^2+
12e^3q^2 (n+1)^{2-4q} \wt\zeta(2-4q) + 4e^2q^2 n\log n~.
\]
\end{lemma}
\begin{proof}
Knowing the parent selectors $p_1,\ldots ,p_n$ and the coin flips $\xi _1 ,\dots ,\xi _n$, we have that $N_0$ is a function of the independent random variables $m_1,\dots ,m_n$. Note that resampling one of them, say $m_i$, does not change the value of $N_i$. Moreover, resampling $m_i$ can change $N_0$ by at most $N_i$: if before resampling we had $m_i=0$ and $\wt{T}_i\subset \wt{T}_0$, and after resampling we have $m_i=1$, then $N_0$ decreases by $N_i$; also, if before resampling we had $m_i=1$ and after resampling we have $m_i=0$, then $\wt{T}_i$ might become a subtree of $\wt{T}_0$ and then $N_0$ increases by $N_i$. Hence, by the Efron-Stein inequality (\cite{EfSt81,Ste86}),
\[
   \var(N_0|p_1,\ldots,p_n,\xi_1,\dots ,\xi_n)\le  \sum_{i=1}^n 2q(1-2q)
   \EXP\left[ N_i^2|p_1,\ldots,p_n,\xi_1,\dots ,\xi_n\right]~.
\]
Hence, writing $Z_0=\EXP \left[N_0|p_1,\dots ,p_n,\xi_1,\dots ,\xi_n\right]$, we have
\[
   \var(N_0) = \EXP \ \var(N_0|p_1,\ldots,p_n,\xi_1,\dots ,\xi_n) + \var(Z_0)
    \le 2q \sum_{i=1}^n 
   \EXP N_i^2+ \var(Z_0)~.
\]
The first term on the right-hand side may be bounded, using Lemma~\ref{lem:12}, by 
\begin{eqnarray*}
2q \sum_{i=1}^n 
   \EXP N_i^2 
& \le &
   2qe^{2}\sum_{i=1}^n \left(\left(\frac{n+1}{i+1}\right)^{2-4q}
     \left(4+e\right)+1\right) \\
& \le & 2qe^2(4+e) (n+1)^{2-4q} \zeta(2-4q) + 2nqe^2~.
\end{eqnarray*}

To bound $\var(Z_0)$,
let $\delta _i$ be the distance between the root and node $i$ in $\wt{T}_0$. These distances are a function of $p_1,\dots ,p_n$ only and, therefore, we have
\[
Z_0=\sum _{v}\left(1-2q\right)^{\delta _v}=1+\sum_{j=1}^n\left(1-2q\right)^{\delta _j}~.
\]
We define
\[
Z_j=\sum _{v\in T^0_{j\downarrow}}\left(1-2q\right)^{\delta _v-\delta_j},\;\;\;0\leq j\leq n~.
\]
Let $Z_i'$ denote the modification of $Z_i$ when the random variable
$p_i$ is replaced by an independent copy $p_i'$ and the other values
$p_1,\dots p_{i-1},p_{i+1},\dots, p_n$ are kept unchanged.
Define similarly the variables $\delta '_i$. Observe that if $p_j$ is replaced by $p'_j$, then
\[
Z_0-Z'_0=Z_j\left(\left(1-2q\right)^{\delta _j}-\left(1-2q\right)^{\delta '_j}\right)
\]
whose absolute value is at most
\[
Z_j\left(1-2q\right)^{\min \left(\delta _j,\delta'_j\right)}\left(1-\left(1-2q\right)^{|\delta _j-\delta'_j|}\right)\leq\left\{
\begin{array}{ll}
0, & \mathrm{if}\;\;\delta _j=\delta ' _j\\
 Z_j2q|\delta _j-\delta'_j|, & \mathrm{else}
\end{array}\right. 
\]
Therefore, by the Efron-Stein inequality,
\begin{align*}
\var\left[Z_0\right] &\leq \frac{1}{2}\sum _{j=1}^n\EXP \left[Z_j^24q^2\left(\delta _j-\delta'_j\right)^2\right]\\
 &= 2q^2\sum _{j=1}^n\EXP \left[Z_j^2\right]\EXP \left[\left(\delta _j-\delta'_j\right)^2\right]\tag*{(by independence)}
\end{align*}
By Jensen's inequality,  $\EXP \left[Z_j^2\right]\leq \EXP
\left[N_j^2\right]$. Moreover, 
\begin{equation}
\EXP \left[\left(\delta_j-\delta'_j\right)^2\right]= 2\var\left[\delta _j\right]\leq
2\log j
\label{eq:2000}\end{equation}
by well-known properties of uniform random
recursive trees 
(Devroye \cite{devroye1988applications}).
Therefore, 
\begin{eqnarray*}
\var\left[Z_0\right] &\leq & 4q^2\sum _{j=1}^n\EXP
                       \left[Z_j^2\right]\log j \\
&\leq & 4q^2\sum _{j=1}^n\EXP
                            \left[N_j^2\right]\log j \\
&\leq & 4q^2\sum  _{j=1}^n\left(\left(\frac{n+1}{j+1}\right)^{2-4q}
        e^2\left(4+e\right)+e^2\right)\log j   \\
& & \quad\text{(by Lemma~\ref{lem:12})}\\
& \le & 12e^3q^2 (n+1)^{2-4q}\sum_{j=1}^n \frac{\log j}{\left( j+1\right)^{2-4q} } + 4e^2q^2 \log(n!)  \\
& \le & 
12e^3q^2 (n+1)^{2-4q} \wt\zeta(2-4q) + 4e^2q^2 n\log n~.
\end{eqnarray*}
\end{proof}

\subsection{Majority of the leaf bits}
\label{sec:leaves}

We have proved Theorem \ref{thm:majority} for the root-bit
reconstruction problem. It remains to show the analogous statements
for the reconstruction problem from leaf bits, that is, for the case
when $\wh{b}^{\text{maj}}$ denotes the majority vote among the bit
values observed on the leaves only. This may be done quite simply, 
as the proof presented in Section \ref{sec:linbound} may be easily 
modified to handle this case. 

Recall that $N_i$ is the maximum number of unmarked vertices in a
subtree rooted at $i$ in $T^0_{i \downarrow}$ 
($i$ is included and can be marked or not marked). 
Let $\ol{N}_i$ be the number of them that are leaves. 
It suffices to show that the first and second moments of $\ol{N}_i$ 
satisfy inequalities analogous to those of Lemmas \ref{lem:10},
\ref{lem:12}, and \ref{lem:14}, with possibly different constants.

The next lemma establishes the desired analogues of Lemmas
\ref{lem:10} and \ref{lem:12}. This suffices to prove
(\ref{eq:majlin}) by the same argument as before.
(The corresponding extension of Lemma
\ref{lem:14} is straightforward and is omitted.)

\begin{lemma}
\label{lem:leafmoments}
For all $i\leq n$,
\[
\frac{1}{32e}\left(\frac{n+1}{i+1}\right)^{1-2q}-\frac{i}{8ne}\leq\EXP
 \left[\ol{N}_i\right]\leq e\left(\frac{n+1}{i+1}\right)^{1-2q}
\]
 and 
\[
\EXP \left[\ol{N}_i^2\right]\leq \left(\frac{n+1}{i+1}\right)^{2-4q}e^{2(1-2q)}\left(4+e\right)+e\left(1-2q\right)~.
\]
\end{lemma}

\begin{proof}
The upper bounds for the expectation and the second moment clearly hold by the fact that
$\ol{N}_i\leq N_i$ and by Lemma~\ref{lem:10}. 

Recall from the proof of Lemma \ref{lem:10} that for $t\in
\{i,\ldots,n-1\}$,
$Y_t$ denotes the number of vertices that are counted by $N_i$ and whose label is at most $t$. Similarly, define $\ol{Y}_t$ as the number
of leaves in the same subtree. 
Hence, $\ol{Y}_n$ is distributed as $\ol{N}_i$. For $t\in \{i+1,\ldots,n\}$, we have
\[
\EXP  \left[ \ol{Y}_t \big|\ol{Y}_{t-1},Y_{t-1} \right]
= \ol{Y}_{t-1}+\frac{1-2q}{t}\left(Y_{t-1}-\ol{Y}_{t-1}\right),
\]
since given $\ol{Y}_{t-1}, Y_{t-1}$, with probability $\frac{1-2q}{t} \left(Y_{t-1}-\ol{Y}_{t-1}\right)$ the number of leaves increases by one ($1-2q$ is the probability that the new vertex is unmarked).
Hence $a_t \defeq \EXP \ol{Y}_t$ satisfies, for $t\in \{i+1,\ldots,n\}$,
\[
a_t = a_{t-1}\left(1-\frac{1-2q}{t}\right)+f(t)~,
\]
 where $f(t) = \frac{1-2q}{t}\EXP Y_{t-1}$.
Solving the recurrence we have 
\begin{align*}
a_n \;\; & \geq  \;\;
\sum_{j=i}^{n-1}f(j+1) 
  \prod_{k=j+1}^n\left(1-\frac{\left(1-2q\right)}{k}\right) \\
\;\; & \geq  \;\;
\sum_{j=i}^{n-1}\frac{1-2q}{e \left(j+1\right)}\left(\frac{j+1}{i+1}\right)^{1-2q}\frac{j}{n}\tag{by Lemma~\ref{lem:10}}
         \\
\;\; & \geq  \;\; \frac{1-2q}{2ne \left(i+1\right)^{1-2q}}\int _{j=i}^{n-1}x^{1-2q}dx \\
\;\; & \geq  \;\;  \frac{1}{8ne \left(i+1\right)^{1-2q}}\left(\left(n-1\right)^{2-2q}-i^{2-2q}\right)\\
\;\; & \geq  \;\;   \frac{1}{32e}\left(\frac{n+1}{i+1}\right)^{1-2q}-\frac{i}{8ne}~.
\end{align*}
\end{proof}

\section{The centroid rule}
\label{sec:centroid}

\subsection{The bit value of the centroid}

In this section we analyze the centroid rule and prove Theorem
\ref{thm:centroid}. The case when only the leaf bits are observed is
discussed in Section \ref{sec:centroidleaves} below. Recall the notation introduced in Section \ref{sec:intro}.

  Assume that the bit value of each vertex is
observed. In this case, $\wh{b}_\text{cent}$ is the bit value of one
of the at most two centroids of the tree.
First notice that, with high probability, the centroid of a uniform random
recursive tree is unique:
\begin{lemma}
\label{lem:uniquecentroid}    
If $T_n$ is a uniform random recursive tree on $n+1$ vertices, then
\[
  \PROB \{ T_n \ \text{has two centroids} \}   = \left\{ \begin{array}{ll}
                 0 & \ \text{if $n$ is even} \\
                  \frac{4}{n+3} & \ \text{if $n$ is odd.} \\
                                                     \end{array} \right.
\] 
\end{lemma}

\begin{proof}

Recall that the number of recursive trees on $n+1$ vertices equals
$n!$ and each of them are equally likely. 

Any tree with an odd number of vertices has a unique centroid, so the
first half of the statement is obvious. For odd $n$, if the tree has
two centroid vertices $L,R$, then there exist two disjoint subtrees of
$(n+1)/2$ vertices, each containing one of the centroids (i.e., there
exists a \textit{central edge} $LR$). Call these subtrees \emph{left} and
\emph{right} subtree. The left subtree contains vertex $L$ and the
right subtree contains vertex $R$. We may assume, without loss of
generality, that the label of $L$ is smaller than the label of
$R$. Then vertex $0$ belongs to the left subtree. Moreover, the two
subtrees correspond to unique recursive trees of $\frac{n+1}{2}$
vertices, after suitable relabelling that respects the relative
ordering of the labels.

To count the number of recursive trees with two centroids, note that
there are $\binom{n+1}{\frac{n-1}{2}}$ ways of choosing the labels in the
left subtree, excluding the label of $L$. Then there are are
$\frac{n-1}{2}+2$ remaining labels. The label of vertex $R$ is smaller than
all its descendants and larger than the label of $L$. Hence $L$ has the
smallest available label and $R$ has the second smallest available
label. Once the labels in the left subtree (and therefore in the right
subtree) are fixed, there are $(\left(\frac{n-1}{2}\right)!)^2$ ways
of selecting the recursive trees that correspond to each. Hence,



\[
\PROB \{ T_n \ \text{has two centroids} \}
= \frac{\binom{n+1}{\frac{n-1}{2}}\cdot
  \left(\left(\frac{n-1}{2}\right)!\right)^2}{n!} =\frac{4}{n+3}~.
\]
\end{proof}

Let $D_n$ (or $D$ when it is clear from the context) be the edge distance between the root and $v^*$ in $T_n$. Then, given $D$, the
number of changes of the bit value on the path between the root and
$v^*$ is Binomial$(D,q)$, independent of $D$. Thus,
\begin{eqnarray*}
\PROB\left\{ \wh{b}_\text{cent}\neq B_0 \right\}
& = & \EXP
      \left[\IND_{\{\text{Binomial}\left(D,q\right) \ \text{is odd} \}}\right]   \\
& = & \frac{1-\EXP
      \left[\left(-1\right)^{\text{Binomial}\left(D,q\right)}\right]}{2}   \\
& = & \frac{1-\EXP \left[\left(1-2q\right)^{D}\right]}{2}~.
\end{eqnarray*}
It is shown by Moon \cite{moon2002centroid} that
  the probability that the root  is a centroid is
  asymptotically positive.
  In particular, Moon proves
$$\liminf_{n\to \infty} \PROB\{ \delta _n=0 \}\rightarrow 1-\log 2~,$$
where $\delta _n$ is the distance between the root and the closest centroid to the root.
Hence,  for all $q \le 1/2$,
\begin{eqnarray}
   \limsup_{n\to \infty} \PROB\left\{ \wh{b}_\text{cent}\neq B_0
    \right\} &\le& \frac{1}{2} -\frac{1}{2} \liminf_{n\to \infty} \PROB\{ D_n=0 \}\label{eq:1002} \\
&= &\frac{1}{2} -\frac{1}{2} \liminf_{n\to \infty} \PROB\{ \delta _n=0 \} \\
&= & \frac{1}{2} -\frac{1}{2} \left(1-\log 2\right)
\label{eq:501}
\end{eqnarray}
proving the second statement of Theorem \ref{thm:centroid}.

To prove the first statement of Theorem \ref{thm:centroid}, note that
\begin{equation}
\frac{1-\EXP \left[\left(1-2q\right)^{D}\right]}{2}
\leq q\EXP D~.
\label{eq:502}
\end{equation}
It follows from Lemma \ref{lem:uniquecentroid} and
  the following result of Moon that $\lim_{n\to \infty} \EXP D=1$.

\begin{theorem}\label{moon} \upshape{(\cite[Theorem 2.1]{moon2002centroid})} Let $\delta _n$ be the depth of the centroid that is closest to the root. Then for any $n\geq 0$,
$$\EXP\left[\delta_n\right]=\bigg\{\begin{array}{ll}
        \frac{n}{n+2} & \text{for $n$ odd}\\
        \frac{n-1}{n+2} & \text{for $n$ even.}
        \end{array}
  $$
\end{theorem}

It follows that in the root-bit reconstruction problem, the
centroid rule satisfies
\[
\limsup_{n\to \infty} R^{\text{cent}} (n,q) \le  q \quad \text{for all} \ q \in [0,1]~.
\]

\subsection{Centroid rule from leaf bits}
\label{sec:centroidleaves}

To complete the proof of Theorem \ref{thm:centroid}, it remains to
consider the reconstruction problem from leaf bits. Recall that in
this case the centroid rule localizes a leaf vertex that is closest to
a centroid and guesses the root bit $B_0$ by the bit value at this
leaf.

The key property for proving the linear upper bound for the asymptotic 
probability of error is the following lemma,  stating that in a
uniform random recursive tree, 
the expected distance of the nearest leaf to the
root is bounded. 

\begin{lemma}
\label{lem:leafdist}
In a uniform random recursive tree $T_n$, define
\[
\Delta _n= \min_{i: \ \text{vertex $i$ is a leaf}} d(i,0)~.
\] 
Then, for all $n$, 
\[
\EXP \Delta _n\le 11  +\mathcal{O}\left(n^{-1-3\log\left(3/e\right)}\right)~.
\]
In particular,
\[
    \limsup_{n\to \infty}  \EXP \Delta _n\le 11~.
\]
\end{lemma} 

\begin{proof}
We write $\Delta =\Delta _n$, and start with the decomposition
\[
    \EXP \Delta \le 2 + 3(\log n) \PROB\{ \Delta > 2\}+ \sum_{i > 3\log n}
    \PROB\{ \Delta \ge i\}~.
\]
To bound $\PROB\{ \Delta > 2\}$, we show that, with probability at least
\[
1- \frac{3}{\log n}\left(1+o_n\left(1\right)\right)~,
\]
 the uniform random recursive tree $T_n$  has a leaf at depth
$2$. Let $A_i$ be the event that $i$ is a leaf, and $B_i$ the event that $d\left(i,0\right)=2$. Let $X=\sum \limits_{i=\lceil 2n/3\rceil}^n\IND _{A_i\cap B_i}$ be the number of leaves at distance 2 from the root, among the vertices $\lceil 2n/3\rceil ,\dots , n$. We bound the mean and variance as follows. 

First, note that $A_i=\bigcap\limits_{j=i+1}^n\left\{p_j\neq i\right\}$ and $B_i=\bigcup\limits_{j=1}^{i-1}\left\{p_i= j,\; p_j=0\right\}$. Then $A_i$ and $B_i$ are independent and 
$$\PROB \left\{A_i\right\}=\prod\limits_{j=i+1}^n\left(1-\frac{1}{j}\right)=\frac{i}{n}\;\;\;\mathrm{and}\;\;\;\PROB \left\{B_i\right\}=\sum\limits_{j=1}^{i-1}\left(\frac{1}{i}\cdot\frac{1}{j}\right)=\frac{H_{i-1}}{i}~.$$
Thus, 
$$\EXP X=\sum\limits_{i=\lceil 2n/3\rceil}^n\left(\frac{1}{n}\cdot \frac{iH_{i-1}}{i}\right)=\left(1+o\left(1\right)\right)\frac{\log n}{3}~.$$
We now turn to the calculation of $\EXP\left\{X^2\right\}$. For $2n/3\leq i<k\leq n$ we have 
$$\PROB\left\{A_k|A_i\right\}=\prod\limits_{l=k+1}^n\PROB\left\{ p_l\neq k|p_l\neq i \right\}=\prod _{l=k+1}^n\left(1-\frac{1}{l-1}\right)=\frac{k-1}{n-1},$$
so
$$\PROB\left\{A_k\cap A_i\right\}=\PROB\left\{A_i\right\}\PROB\left\{A_k\right\}\frac{\left(k-1\right)n}{k\left(n-1\right)}=\left(1+\mathcal{O}\left(\frac{1}{n}\right)\right)\PROB\left\{A_i\right\}\PROB\left\{A_k\right\}~.$$
Moreover, $\PROB\left\{B_i\cap B_k| A_i\cap A_k \right\}= \PROB\left\{B_i\cap B_k|p_k\neq i\right\}$, which is equal to
\begin{eqnarray*}
& & \sum _{j=1}^{i-1}\PROB\left\{p_i=p_k=j,p_j=0|p_k\neq i\right\}+\sum _{j=1}^{i-1}\sum\limits_{\substack{l=1\\ l\neq j}}^{k-1}\PROB\left\{p_i=j,p_j=0\right\}\PROB\left\{p_k=l,p_l=0|p_k\neq i\right\}\\
&=&\frac{1}{k-1}\cdot \frac{H_{i-1}}{i}+\sum _{j=1}^{i-1}\sum\limits_{\substack{l=2\\ l\neq j}}^{k-1}\PROB\left\{p_i=j,p_j=0\right\}\PROB\left\{p_k=l,p_l=0|p_k\neq i\right\}~.
\end{eqnarray*}
Since $k\geq 2n/3$, we have
$$\PROB\left\{p_k=0,p_l=0|p_k\neq i\right\}=\frac{1}{k-1}\cdot \frac{1}{l-1}=\left(1+o\left(\frac{1}{n}\right)\right)\PROB\left\{p_k=l,p_l=0\right\}~.$$
To handle the 
$j=l$ term, we note that 
$$\sum _{j=2}^{i-1}\PROB\left\{p_i=j,p_j=0\right\}\PROB\left\{p_k=j, p_j=0\right\}=\frac{1}{k\cdot i}\sum_{j=1}^{i-1}\frac{1}{j^2}=\mathcal{O}\left(1\right)\cdot\frac{1}{k\cdot i}~.$$
It follows that 
$$\PROB\left\{B_i\cap B_k|A_i\cap A_k\right\}=\left(1+\mathcal{O}\left(\frac{1}{n}\right)\right)\PROB\left\{B_i\right\}\PROB\left\{B_k\right\}+\frac{H_{i-1}-\mathcal{O}\left(1\right)}{i\left(k-1\right)},$$
So, recalling that $A_i$ and $B_i$ are independent for all $i$, $\EXP\left[X^2\right]$ is equal to
\begin{eqnarray*}
&& \sum _{2n/3\leq i\leq n}\sum _{2n/3\leq k\leq n}\PROB\left\{A_i\cap B_i\cap A_k\cap B_k\right\}\\
&=& \sum\limits_{2n/3\leq i\leq n}\PROB\left\{A_i\cap B_i\right\}  \\
& & +2\sum\limits_{2n/3\leq i<k\leq n}\left[ \left(1+\mathcal{O}\left(\frac{1}{n}\right)\right)\PROB\left\{A_i\cap B_i\right\}\PROB\left\{A_k\cap B_k\right\}+\PROB\left\{A_i\cap A_k\right\}\frac{H_{i-1}-\mathcal{O}\left(1\right)}{i\left(k-1\right)}  \right]\\
&\leq &\left(1+\mathcal{O}\left(\frac{1}{n}\right)\right)\left(\left(\EXP X\right)^2+\sum\limits_{2n/3\leq i\leq n}\left(\PROB\left\{A_i\cap B_i\right\}-\PROB\left\{A_i\cap B_i\right\}^2\right) \right)+o\left(1\right)\\
&\leq &\left(\EXP X\right)^2+\frac{1}{3}\log n\left(1+o\left(1\right)\right)
\end{eqnarray*}
Recalling that $\EXP X=\left(1+\mathcal{O}\left(\frac{1}{n}\right)\right)\frac{\log n}{3}$, it follows that 
$$\PROB\left\{X=0\right\}\leq\frac{\var \left\{X\right\}}{\left(\EXP \left\{X\right\}\right)^2}\leq\frac{3\left(1+o_n\left(1\right)\right)}{\log n}~.$$

It remains to bound $\sum_{i > 3\log n} \PROB\{ \Delta \ge i\}$. We do
this simply by bounding $\Delta$ by $d(n,0)$, the depth of vertex $n$.
By standard results on uniform random recursive trees (see Devroye \cite{devroye1988applications}), the
 insertion depth $d(i,0)$ of vertex $i$ is distributed as $\sum_{j=1}^i Y_j$,
 where the $Y_j$ are independent Bernoulli random variables with
 $\PROB\{Y_j=1\} =1/j$.
 By the standard Chernoff bound for sums of independent Bernoulli variables~\cite[Exercise 2.10]{BoLuMa13}, 
 \begin{equation}
 \label{eq:chernoff}
    \PROB\left\{ d(i,0)\ge t \right\} \le \exp\left( t - H_i -
      t\log\frac{t}{H_i} \right)~,
 \end{equation}
 where $H_i=\sum_{j=1}^i 1/j $.
By (\ref{eq:chernoff}) above, for all $i>3H_n$,
\[
\PROB\{ \Delta \ge i\} \le  \PROB\left\{ d(n,0)\ge i \right\} \le \exp\left( i - H_n -
     i\log\frac{i}{H_n} \right)
  \le \frac{1}{n} e^{-i\log(3/e)}~.
\]
Thus,
\[
    \sum_{i > 3\log n} \PROB\{ \Delta \ge i\} =\mathcal{O}\left(n^{-1-3\log\left(3/e\right)}\right)~.
\]
Collecting terms, the proof of the lemma is complete.\end{proof}

If $\wt{v}$ is a leaf vertex that is closest to the centroid $v^*$,
then its distance to the root is bounded as follows.
\[
 d(\wt{v},0)\le d(\wt{v},v^*)+d(0,v^*)\leq d(0,\wt{v})+2d(0,v^*)=\Delta + 2D~,
\]
where $D=d(v^*,0)$. 
Hence,  Lemmas \ref{lem:uniquecentroid},~\ref{moon}, and \ref{lem:leafdist}
imply that
\[
     \limsup_{n\to \infty} \EXP d(\wt{v},0) \le 13~,
\]
proving the third statement of Theorem \ref{thm:centroid}.

\section{The case $q>\frac{1}{2}$}
\label{sec:greaterthanhalf}

In this section we finish the proof of Theorem~\ref{thm:main2} by showing that $R^*\left(q\right)<1/2$ even when $\frac{1}{2}<q<1$.   
The main idea is that, with probability bounded away from zero, the URRT has a certain structure and, if this
structure happens to occur, then the root can be identified with probability greater than $1/2$. 
Then one may proceed by identifying if the given structure occurs. If it doesn't, one may toss a random coin.
If it does, one tries to identify the root and picks the associated bit value. 

First we show that this strategy works when the bit values associated to all vertices are observed. 
Since the vertex identified as root is not a leaf, this strategy does not work in the reconstruction problem
from leaf bits. However, an easy modification works when only bit values on the leaves are available. This is shown
in Section \ref{sec:abovehalfleaves}.

\subsection{Root-bit reconstruction}
\label{sec:rootbitabovehalf}

The structure of the URRT that we require is described in Definition~\ref{def:1}. Recall the definitions
of Aut and $\ol{\text{Aut}}$ from the introduction.

\begin{figure}
\centering
\includegraphics[scale=.9]{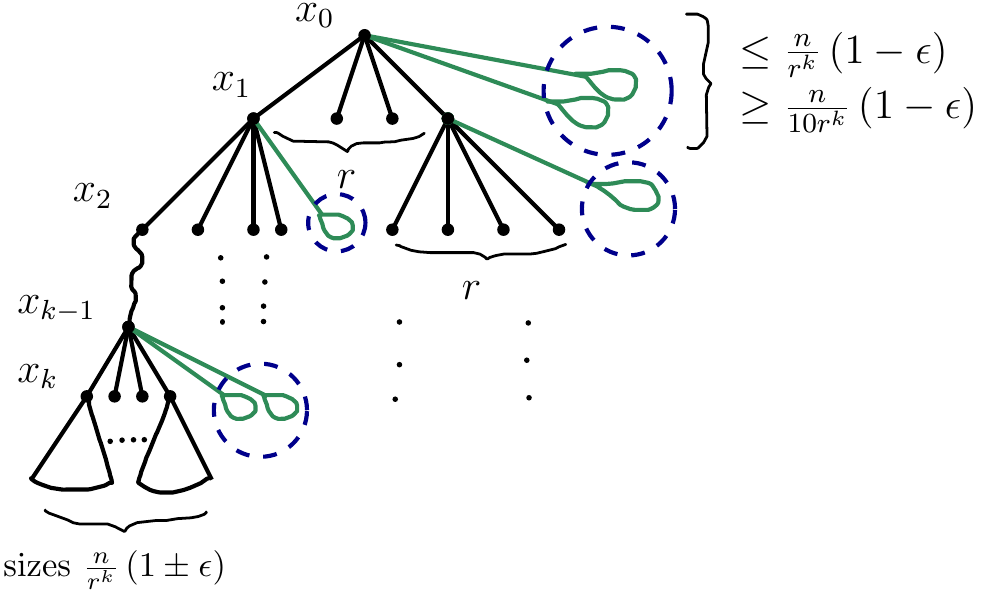}
\label{picture}
\caption{A depiction of condition (II) of the event $E_{r,k}$, described in Definition~\ref{def:1}.}
\end{figure}

\begin{definition}(see also Figure~\ref{picture})
\label{def:1}
Fix integers $r,k>3$ such that $k\leq r$ and let $\epsilon \in (0,\frac{1}{2r^k})$. 
Let $E_{r,k}$ denote the event that the following conditions are satisfied:
\begin{enumerate}[(I)]
\item $T_n$ contains a complete rooted $r$-ary subtree $D$ of height $k$ (we denote its root-vertex by $x_0$ and its leaves by  $L\left(D\right)$). 
\item Let $T$ be any subtree of $T_n$ which is maximal subject to the constraint that $\left| T\cap D\right|=1$, and write $v$ for the unique vertex of $T\cap D$. If $v\in D\setminus L\left(D\right)$ then
$T$ has at most $\left(1-\epsilon\right)\frac{n}{r^k}$ vertices and at least $\left(1-\epsilon\right)\frac{n}{10r^k}$ vertices. 
If $v\in L\left(D\right)$, then $T$ has at most $\left(1+\epsilon\right)\frac{n}{r^k}$ vertices and at least $\left(1-\epsilon\right)\frac{n}{r^k}$ vertices.
 \item  All maximal subtrees that intersect $D$ on exactly one vertex which has depth $k$ (in $D$) are different as unlabelled rooted trees. 
 \item For all $v\in D\setminus L\left(D\right)$, $\text{Aut}\left(v,T_n\right)=\ol{\text{Aut}}\left(T^{x_0}_{v\downarrow}\right)=1$.
 \end{enumerate}
 \end{definition}
 
We now present the skeleton of the proof. Some of the technical details are deferred to later. 
 
\begin{proof} {\bf (Theorem~\ref{thm:main2}, case $q>\frac{1}{2}$.)}
Recall that $x_0$ is the root vertex of $D$. Fix $r,k>3$ such that $k\leq r$ and fix $\epsilon \in (0,\frac{1}{2r^k})$. 
Let $p_i=(1/2)\left(1+\left(1-2q\right)^i\right)$ be the probability that a vertex at distance $i$ from the root $0$ has the same bit value $B_0$ as the root and denote $\ol{D}:=D\setminus L\left(D\right)$. Then we have 
\begin{eqnarray*}
\lefteqn{
\PROB\left\{ B_{x_0}=B_0\big|E_{r,k}\right\}   } \\
& \ge &\sum _{i=0}^{k-1}\PROB\left\{ B_{x_0}=B_0\big|E_{r,k},\;0\in\ol{D},\;d\left(0 ,x_0\right)=i\right\} 
\PROB\left\{ 0 \in\ol{D},\;d\left(0 ,x_0\right)=i \big|E_{r,k}\right\} \\
&\geq & 
\exp \left(-\frac{k}{r^k}\right)\left(1-\frac{1}{r^{k-1}}\right)^2\frac{\sum _{i=0}^{k-1} p_i
r^i\prod _{j=1}^i\frac{1}{r^j-1}}{\sum _{i=0}^{k-1}r^i\prod _{j=1}^i\frac{1}{r^j-1}}+o_n\left(1\right)     \\
& &
\qquad \text{(by Lemma~\ref{prob} below)}\\  
&= & \exp \left(-\frac{k}{r^k}\right)\left(1-\frac{1}{r^{k-1}}\right)^2\frac{\sum _{i=0}^{k-1}\frac{1}{2}\left(1+\left(-1\right)^i(2q-1)^i\right)
r^i\prod _{j=1}^i\frac{1}{r^j-1}}{\sum _{i=0}^{k-1}r^i\prod _{j=1}^i\frac{1}{r^j-1}}+o_n\left(1\right)     \\
\end{eqnarray*}

\noindent Note that 
\[
\sum _{i=0}^kr^i\prod_{j=1}^i\frac{1}{r^j-1}=1+\frac{r}{r-1}+\frac{r^2}{\left(r-1\right)\left(r^2-1\right)}+\dots=2+\mathcal{O}\left(\frac{1}{r}\right)
\]
and 
\begin{eqnarray*}
\sum _{i=0}^k\left(-1\right)^i\left((2q-1)r\right)^i\prod_{j=1}^i\frac{1}{r^j-1}
&=& 1-\frac{(2q-1)r}{r-1}+\frac{(2q-1)^2r^2}{\left(r-1\right)\left(r^2-1\right)}+\dots\\ &=&1-(2q-1)+\mathcal{O}\left(\frac{1}{r}\right)~,
\end{eqnarray*}
and therefore
\begin{eqnarray*}
\liminf_{n\to \infty} \PROB\left\{ B_{x_0}=B_0\big|E_{r,k}\right\}
&\geq & 
\exp\left(-\frac{k}{r^k}\right)\left(1-\frac{1}{r^{k-1}}\right)^2\left(\frac{1}{2}+\frac{1}{2}\cdot\frac{1-(2q-1)+\mathcal{O}\left(\frac{1}{r}\right)}{2+\mathcal{O}\left(\frac{1}{r}\right)}\right)\\
& = & \frac{3}{4}-\frac{2q-1}{4}+\mathcal{O}\left(\frac{1}{r}\right)  \\
& = & \frac{2-q}{2}+\mathcal{O}\left(\frac{1}{r}\right)>\frac{1}{2}
\end{eqnarray*}
for large enough $r$. Since $\liminf_{n\to \infty}
\PROB\left\{E_{r,k}\right\} >0$ by Lemma~\ref{positive} below, 
there exists a  choice of the parameters $r$ and $k$ (depending on $q$
only) such that the procedure that 
 guesses $B_{x_0}$ if the event $E_{r,k}$ occurs and guesses a random bit otherwise
is positively correlated with $B_0$.
\end{proof}
 
It remains to prove the two key properties used in the proof above.

\begin{lemma}
\label{prob}
Let $r,k>3$ with $k\leq r$ and  let $\epsilon\leq \frac{1}{2r^k}$. Then for all $i=0,1,\ldots,k-1$,
\[
\liminf_{n\to \infty} \PROB\left\{ 0 \in\ol{D},\;d\left(0 ,x_0\right)=i|E_{r,k}\right\} 
\geq \exp \left(-\frac{k}{r^k}\right) \left(1-\frac{1}{r^{k-1}}\right)^2
\frac{r^i\prod_{j=1}^i\left(\frac{1}{r^j-1}\right)}{\sum_{m<k} r^m\prod_{j=1}^m\left(\frac{1}{r^j-1}\right)}~.
\]
\end{lemma}
\begin{proof}
We first lower bound $\PROB\left\{ 0 \in\ol{D}|E_{r,k}\right\}$. Notice that under the event $E_{r,k}$,
if $0\not\in\ol{D}$, then either $T^0_{1\downarrow}$ contains at least 
$n\left(1-(1-\epsilon)/(10r^k)\right)$ 
vertices or it contains at most $\left(1+\epsilon\right)n/r^k$ vertices. 
By standard results of the theory of P\'olya urns (Eggenberger and P\'olya \cite{EgPo23}),
$\left|T^0_{1\downarrow}\right|$ converges, in distribution, to a uniform random variable on $[0,1]$.

Hence, 
\begin{eqnarray*}
\PROB\left\{ 0 \in\ol{D}|E_{r,k}\right\}
& = & 1-\frac{1-\epsilon}{10r^k} - \frac{1+\epsilon}{r^k}+o_n\left(1\right)  \\
& \geq &  1-\frac{2}{r^k}+o_n\left(1\right) \ge   1-\frac{1}{r^k-1}+o_n\left(1\right)~.
\end{eqnarray*}
It remains to derive a lower bound for 
\[
\PROB\left\{ d(0 ,x_0)=i\big|0 \in\ol{D},\;E_{r,k}\right\} =
\sum_{v\in \ol{D}: d(v,x_0) =i}\PROB\left\{  0=v  \big|0 \in\ol{D},\;E_{r,k}\right\}~.
\] 
Recall the definition of the function $\lambda(u)$ from (\ref{eq:likelihood}) and that, given an unlabeled tree, the probability 
that vertex $u$ is the root is proportional to $\lambda(u)$.
Hence, defining, for $i=0,1,\ldots,k-1$,
\[
      W_i = \sum_{v\in \ol{D}: d(v,x_0) =i} \frac{\lambda(v)}{\lambda(x_0)}~,
\]
we have that 
\[
\PROB\left\{ d(0 ,x_0)=i\big|0 \in\ol{D},\;E_{r,k}\right\} =
\frac{W_i}{\sum_{j=0}^{k-1} W_j}~.
\]



Under the event $E_{r,k}$, for all $u\in\ol{D}$, we have $\text{Aut}\left(u,T\right)=1$ and $\ol{\text{Aut}}\left(T_{u\downarrow}\right)=1$. 
Hence, if $x_i\in\ol{D}$ has depth $i$ in $D$ and $x_0x_1\dots x_i$ is the path in $D$ that connects it to the root of $D$, then,
for all $j=1,\ldots,i-1$,
\[
\frac{\lambda(x_{j+1})}{\lambda(x_j)}
\geq \frac{\frac{n}{r^j}\left(1-\epsilon\right)}{n-\frac{n}{r^j}\left(1-\epsilon\right)}
=\frac{1}{r^j-1}\left(1-\frac{\epsilon r^j}{r^j-1+\epsilon}\right)\geq \frac{1}{r^j-1}\left(1-\frac{1}{r^k}\right)~, 
\]
since $\epsilon\leq \frac{1}{2r^k}$.
Thus,
\begin{eqnarray*}
\frac{\lambda(x_i)}{\lambda(x_0)}
&  \geq & 
\left(1-\frac{1}{r^k}\right)^k\prod_{j=1}^i\left(\frac{1}{r^j-1}\right) \\
& \geq  & \left(1-\frac{k}{r^k}\right)\prod_{j=1}^i\left(\frac{1}{r^j-1}\right)\geq \left(1-\frac{1}{r^{k-1}}\right) \prod_{j=1}^i\left(\frac{1}{r^j-1}\right)~.
\end{eqnarray*}
Similarly,
\[
\frac{\lambda(x_{j+1})}{\lambda(x_j)} \leq \frac{\frac{n}{r^j}\left(1+\epsilon\right)}{n-\frac{n}{r^j}\left(1+\epsilon\right)} \leq \left(\frac{1}{r^j-1}\right)\left(1+\frac{1}{r^k}\right) 
\]
and 
\begin{eqnarray*}
\frac{\lambda(x_i)}{\lambda(x_0)}
& \leq & \left(1+\frac{1}{r^k}\right)^k\prod_{j=1}^i\left(\frac{1}{r^j-1}\right)\\
&\leq & \exp \left(\frac{k}{r^k}\right)\prod_{j=1}^i\left(\frac{1}{r^j-1}\right)~.
\end{eqnarray*}
Putting these estimates together, we obtain the statement of the lemma.
\end{proof}


The last ingredient is the following lemma.

\begin{lemma}
\label{positive}
Let $r,k> 3$. Then $\liminf_{n\to \infty} \PROB\left\{ E_{r,k}\right\} > 0$.
\end{lemma}

\begin{proof}
Fixed $k$ and $r$. After the insertion of $M \defeq
\frac{r^{k+1}-1}{r-1}$ vertices, the probability that the uniform
random recursive tree 
 $T_M$ is isomorphic to a complete $r$-ary tree $D$
of height $k$ is a positive value,  depending on $r$ and $k$ only.
Call this event $E_I$. This event clearly implies property $(I)$ in
Definition \ref{def:1}.

In what follows, we condition on event $E_I$. 
Let $u_1,\ldots ,u_{r^k}$ be the vertices of height $k$ in $D$ and $v_1,\dots ,v_{m}$
be the rest of the vertices in $D$. For every such vertex $v_i$ (or
$u_j$ accordingly), we define $\overline{T}_{v_i\downarrow}^{x_0}$ to
be the maximal subtree of $T_{v_i\downarrow}^{x_0}$ that intersects
$D$ in only at $v_i$. Then the vector
$\left(\left|\overline{T}_{v_1\downarrow}^{x_0}\right|,\dots
  ,\left|\overline{T}_{v_{m}\downarrow}^{x_0}\right|,\left|\overline{T}_{u_1\downarrow}^{x_0}\right|,\dots
  ,\left|\overline{T}_{u_{r^k}\downarrow}^{x_0}\right|\right)$ behaves
as a standard P\'olya urn with $M$ colors, initialized with one ball
of each color. As $n$ goes to infinity, the proportions of the
balls of each color converge to a
$\text{Dirichlet}\left(1,\dots,1\right)$ distribution. 

Let 
\begin{eqnarray*}
\Omega  & = & \left\{(x_1,\ldots,x_{M-1}) \in \R^{M-1}: \sum_{i=1}^{M-1} x_i
  =1,       \right. \\
& & \quad \left. x_1,\dots ,x_{r^k}\in
  \left(\frac{1-\epsilon}{r^k},\frac{1-\epsilon/2}{r^k}\right),x_{r^k+1},\dots
  , x_{M-1}\in
  \left(\frac{\epsilon/10}{M-r^k},\frac{\epsilon/2}{M-r^k}\right)\right\}~.
\end{eqnarray*}
Then
\begin{eqnarray*}
\PROB\left\{ \left(II\right)|E_I\right\}
&\geq &
\Gamma \left(M\right)\underbrace{\int _{\frac{1-\epsilon}{r^k}}^{\frac{1-\epsilon/2}{r^k}}\dots \int _{\frac{1-\epsilon}{r^k}}^{\frac{1-\epsilon/2}{r^k}}}_{r^k\;\text{times}}\underbrace{\int _{\frac{\epsilon/10}{M-r^k}}^{\frac{\epsilon/2}{M-r^k}}\dots \int _{\frac{\epsilon/10}{M-r^k}}^{\frac{\epsilon/2}{M-r^k}}}_{M-r^k-1\;\text{times}}dx_{M-1}\dots dx_1+o_n\left(1\right)\\
& = & \Gamma \left(M\right) \left(\frac{\epsilon /2}{r^k}\right)^{r^k}\left(\frac{2\epsilon /5}{M-r^k}\right)^{M-r^k-1}+o_n\left(1\right)~,
\end{eqnarray*}
and therefore properties $(I)$ and $(II)$ jointly hold with 
probability bounded away from zero.

Conditioning on event $E_I$, property $(III)$ of Definition
\ref{def:1}  clearly holds with probability converging to one, since $r,k$ are fixed.

Finally, we check property $(IV)$, conditioned on the properties
$(I),(II),(III)$. We abbreviate $A =(I) \cap (II) \cap (III)$. 
Let $v\in \overline{D}$ and
$S_1,\dots ,S_k$ the subtrees of $T_n$ that are contained in
$T_{v\downarrow}^{x_0}$ and whose roots are connected with an edge to
$v$. 
Denote by $n_v$ the number of vertices of the subtree
$\overline{T}^{x_0}_{v\downarrow}$.
By property $(II)$, $n_v=\Omega(n)$.

We call an \emph{$S_iS_j$-conflict} the event where $S_i\cong S_j$
as rooted unlabelled trees. Moreover, we denote by
$C_i^{\left(n_v\right)}$ the number of indices $j$ such that
$\left|S_j\right|=i$. To finish the proof it suffices to
show that 
\[
\liminf_{n\to \infty} \PROB\left\{\mathrm{ no \;
    S_iS_j\text{-conflict} } | A \right\} >0~.
\]
To this end, it suffices that
\[
\liminf_{n_v\to \infty} 
\left(\PROB\left\{\forall i\leq \sqrt{n_v},\;
    C_i^{\left(n_v\right)}\leq 1  | A  \right\} 
    - \PROB\left\{ \exists  S_iS_j\text{-conflict} \;\mathrm{where}\; \left|S_i\right|>\sqrt{n_v}  | A \right\}\right)>0~.
\]
By independence and since $r,k$ are fixed the claim then holds for all $v\in \overline{D}$ with constant probability.

We need the following claim:

\begin{claim}
\label{claim}
 For any $j>\sqrt{n_v}$, 
\[
\PROB\left\{ C_j^{\left(n_v\right)}\geq 2 | A \right\}
\leq \PROB\left\{ C_{\sqrt{n_v}}^{\left(n_v\right)}\geq 2
 \right\} +\mathcal{O}\left(n_v^{-3/2}\right)~.
\]
\end{claim}

\begin{proof}
The multiset $\left\{|S_1|,\dots , |S_k|\right\}$ is distributed as
the multiset of cycle lengths of a uniformly random permutation of
$\left|T_{v\downarrow}^{x_0}\right|-1$. Hence, by Arratia, Barbour,
and Tavar{\'e} ~\cite[Lemma 1.2]{arratia2003logarithmic},
\begin{equation}
\label{eq:4}
\PROB\left\{ C_j^{\left(n_v\right)}=m | A \right\}  
=\frac{1}{j^{m}m!}\sum_{\ell =0}^{\lfloor  n_v/j\rfloor-m}\frac{\left(-1\right)^\ell}{j^{\ell}\ell !}~.
\end{equation}
Then
\begin{eqnarray*}
\PROB\left\{ C_j^{\left(n_v\right)} \geq 2 | A \right\}
& = & \sum _{m\geq 2}\frac{1}{j^{m}m!}\sum _{\ell=0}^{\lfloor
      n_v/j\rfloor-m}\frac{\left(-1\right)^\ell}{j^{l}\ell!} \\
& < & \sum _{m\geq 2}\left(\frac{1}{\sqrt{n_v}^{m}m!}\left(\sum _{\ell=0}^{\lfloor \sqrt{n_v}\rfloor-m}\frac{\left(-1\right)^\ell}{j^{\ell}\ell!}-\sum _{\ell=\lfloor n_v/j\rfloor-m+1}^{\lfloor \sqrt{n_v}\rfloor-m}\frac{\left(-1\right)^\ell}{j^{\ell}\ell!}\right)\right)\\
& = &  \PROB\left\{ C_{\sqrt{n_v}}^{\left(n_v\right)}\geq 2 | A \right\}+\sum _{m\geq 2}\frac{1}{\sqrt{n_v}^{m}m!}\sum _{\ell=\lfloor n_v/j\rfloor-m+1}^{\lfloor \sqrt{n_v}\rfloor-m}\frac{\left(-1\right)^{\ell+1}}{j^{\ell}\ell!}\\
&\leq & \PROB\left\{ C_{\sqrt{n_v}}^{\left(n_v\right)}\geq 2 | A \right\}+\frac{1}{n_v}\sum _{m\geq 2}\frac{1}{m!}\sum _{\ell= 1}^{\lfloor \sqrt{n_v}\rfloor}\frac{1}{j^{\ell}\ell!}\\
&\leq & \PROB\left\{ C_{\sqrt{n_v}}^{\left(n_v\right)}\geq 2 |A \right\}
  + \frac{e}{n_v}\left(\frac{1}{\sqrt{n_v}}+\frac{1}{n_v}+\dots\right)\\
& = & \PROB\left\{ C_{\sqrt{n_v}}^{\left(n_v\right)}\geq 2 |A \right\}
      + \mathcal{O}\left(n_v^{-3/2}\right)~,
\end{eqnarray*}
and the claim follows.
\end{proof}

\noindent Let $\left(Z_1,\dots ,Z_{n_v}\right)$ be a vector of independent Poisson variables $Z_i$ with mean $\frac{1}{i}$. It is known (see for instance~\cite[Lemma 1.4]{arratia2003logarithmic}) that 
\begin{equation}
\label{abt}
\mathrm{d}_{TV}\left(\left(C_1^{\left(n_v\right)},\dots ,C_b^{\left(n_v\right)}\right),\left(Z_1,\dots ,Z_b\right)\right)\leq \frac{2b}{n_v+1},
\end{equation}
where $\mathrm{d}_{TV}$ denotes the total variation distance. Then, 
\begin{eqnarray*}
\PROB\left\{ \forall i\leq \sqrt{n_v},\; C_i^{(n_v)} \leq 1  \right\} 
&\geq & \prod_{i\leq \sqrt{n_v}}\PROB\left\{
        \text{Poisson}\left(\frac{1}{i}\right)\leq 1\right\}
        -\frac{2\sqrt{n_v}}{n_v+1}
       \quad \text{(by~(\ref{abt}))}\\
&= & \prod_{i\leq \sqrt{n_v}}\exp\left(-\frac{1}{i}\right)\left(1+\frac{1}{i}\right)-\frac{2\sqrt{n_v}}{n_v+1}\\
&\geq & \exp\left(-\log\left(\sqrt{n_v}+1\right)\right)\left(\sqrt{n_v}+1\right)-\frac{2\sqrt{n_v}}{n_v+1}=1-\frac{2\sqrt{n_v}}{n_v+1}
\end{eqnarray*}
and
\begin{eqnarray*}
\lefteqn{
\PROB\left\{ \exists  S_iS_j\text{-conflict} \;\mathrm{with}\;
  \left|S_i\right|>\sqrt{n_v}  | A\right\}     } \\
&\leq & 
\sum _{k>\sqrt{n_v}}\PROB\left\{ C_k^{\left(n_v\right)}\geq 2 |A\right\} \\
&\leq & 
\sum _{k>\sqrt{n_v}} \PROB\left\{  C_{\sqrt{n_v}}^{\left(n_v\right)}\geq 2 |A \right\}
+\mathcal{O}\left(n_v^{-1/2}\right)
\quad \text{(by Claim \ref{claim})}\\
& \leq &  n_v\sum _{m= 2}^{m=\sqrt{n_v}}\frac{1}{\sqrt{n_v}^{m}m!}\sum
         _{\ell=0}^{\lfloor
         \sqrt{n_v}\rfloor-m}\frac{\left(-1\right)^\ell}{\sqrt{n_v}^{\ell}\ell!}
         +\mathcal{O}\left(n_v^{-1/2}\right)
\quad \text{(by~(\ref{eq:4}))}\\
&\leq  & \mathcal{O}\left(n_v^{-1}\right)+\mathcal{O}\left(n_v^{-1/2}\right)~.
\end{eqnarray*}
We may now conclude that for large $n$, for all $v\in \ol{D}$, $\ol{\text{Aut}}\left(T^{x_0}_{v\downarrow}\right)=1$ with constant probability. 

Finally, the constraints on the subtree sizes from (ii) imply that any automorphism of $T_n$ restricts to an automorphism of $D$. It follows that when (ii) holds, for any $v\in D\setminus L\left(D\right)$, any automorphism $\phi$ of $T_n$ with $\phi \left(v\right)\neq v$ must permute the set of subtrees of $T_n$ which intersect $L\left(D\right)$ in exactly one vertex. It follows that if (i),(ii) and (iii) all hold, then no such automorphism can exist, i.e., $\text{Aut}\left(v,T_n\right)=1$. 
\end{proof}

\subsection{Reconstruction from leaf bits}
\label{sec:abovehalfleaves}

The only missing bit from the complete proof of Theorem
\ref{thm:main2} is to show that for $q>1/2$ one may beat random
guessing even when only the leaf bits are observed. This follows
quite easily from the construction of Section
\ref{sec:rootbitabovehalf}.
The  method of the previous section does not work since even when 
the tree $T_n$ has the structure described in Definition \ref{def:1},
the
root of the complete $r$-ary subtree $D$ is not a leaf and therefore
its bit value is not observable. However, it is easy to see that the root 
of a URRT is attached to a leaf with probability bounded away from
zero
(see, e.g., Arratia, Barbour, and Tavar{\'e}
~\cite{arratia2003logarithmic}).
Hence, the following method is easily shown to have a probability of
error bounded away from $1/2$: 

Choose $r$ and $k$ as in the proof in Section
\ref{sec:rootbitabovehalf}.
Let $E'_{r,k}$ be the event that the four conditions listed in
Definition \ref{def:1} are satisfied and moreover a leaf $v$ of $T_n$ is
attached
to the root of the subtree $D$. Now guess the bit value $B_0$ by
flipping the bit value $B_v$ of the leaf $v$. Since $\liminf_{n\to \infty}
\PROB\{E'_{r,k}\} >0$ and the root of $D$ is positively correlated
with $B_0$, we have that
\[
  \liminf_{n\to \infty} \PROB\left\{ 1-B_v = B_0 \right\} >\frac{1}{2}~,
\]
as desired.

\section{Preferential attachment}
\label{sec:pref}

In this section we extend several of our results to the linear preferential model defined in the introduction.  
As most of the arguments are analogous to those of the uniform attachment model, we only give sketches
of the proofs, relegating some of the technical details to the Appendix.

\subsection{The majority rule}

We begin by analyzing the majority rule. Just like in the case of uniform attachment, the asymptotic
probability of error is bounded by a constant multiple of $q$ both in
the root-bit reconstruction problem and in the reconstruction problem from leaf bits. 
Interestingly, the break-down point of the majority rule is not at $q=1/4$ anymore.
The critical value depends on the parameter $\beta$ and it is given by
\[
\gamma(\beta)=\min\left(\frac{\beta+1}{4\beta},\frac{1}{2}\right)~.
\]
Note that this value is always larger than $1/4$ and therefore the majority rule
has a better break-down point than in the case of uniform attachment, for all values of $\beta$.
Moreover, when $\beta \le 1$, the majority vote has a nontrivial probability of error
for all values of $q<1/2$.

\begin{theorem}
\label{thm:majority2}
Consider the broadcasting problem in the linear preferential
attachment model with parameter $\beta >0$.
For both the root-bit reconstruction problem and the reconstruction problem from leaf bits, there exists a constant $c$ such that 
\[
\limsup_{n\to \infty} R^{\text{maj}} (n,q) \le cq \quad \text{for all} \ q \in [0,1]~.
\]

\noindent Moreover, 
\[
\limsup_{n\to\infty} R^{\text{maj}} (n,q) < 1/2 \quad \text{if } \ q \in [0,\gamma(\beta) )~,
\]
and
\[
\limsup_{n\to\infty} R^{\text{maj}} (n,q) = 1/2 \quad \text{if } \ q \in [\gamma(\beta) ,1/2]~.
\]
\end{theorem}

The proof of the linear bound follows exactly the same steps as the corresponding proof of Theorem~\ref{thm:majority}, only here Lemmas~\ref{lem:200},~\ref{lem:201} (shown in Section \ref{sec:prefmaj}  of the Appendix) take the role of Lemmas~\ref{lem:10},~\ref{lem:12},~\ref{lem:leafmoments}. Note that the bound on $\var(\delta _j)$ in~\eqref{eq:2000} that is used in the proof of Lemma~\ref{lem:14}, is similar in the preferential attachment model (see for instance~\cite[Theorem 2.7, Section 7]{drmota2009height}). Hence we omit this proof for brevity.

For the other two assertions, the proof follows the same steps as in Section~\ref{sec:q=1/4}, and Section~\ref{sec:123}, only now the matrix we use encodes the expected change of the \emph{weight} of each of the four categories of nodes. 
The weight of a set $A$ of vertices is defined by $\beta \left| A\right|+\sum _{v\in A}D_v^+$.
We obtain the following matrix:
\begin{equation*}
\left( {\begin{array}{cccc}
   -\beta q &  \beta\left(1-q\right) & \beta q & \beta q   \\
    \beta+1 & 1 & 0 & 0  \\   
     \beta q &   \beta q & -\beta q & \beta\left(1-q\right) \\  
         0 & 0 & \beta+1 & 1 \\  
 \end{array} } \right)
\end{equation*}
The eigenvalues of the transpose of this matrix are $\beta +1, \beta +1-2\beta q,-\beta ,-\beta$ and then  \cite[Theorems~
3.23,~3.24]{janson2004functional} can be immediately applied as before, in combination with Lemmas~\ref{lem:200} and ~\ref{lem:201}.

\subsection{The centroid rule}

For the performance of the centroid rule, we have the following analog of Theorem \ref{thm:centroid} for linear preferential attachment trees.
The proof parallels the arguments of Section~\ref{sec:centroid}. The details are given in Section \ref{sec:prefcent} in the Appendix.

\begin{theorem}
\label{thm:centroid2}
Consider the broadcasting problem in the linear preferential
attachment model with fixed parameter $\beta >0$.
For both the root-bit reconstruction problem and the reconstruction problem from leaf bits, there exists a constant $c$ such that
\[
\limsup_{n\to \infty} R^{\text{cent}} (n,q) \le cq  \quad \text{for all} \ q \in [0,1]~.
\]
In particular, $c\leq \frac{\beta}{\beta +1}$ in the root-bit reconstruction problem and $c\leq 2+\frac{2\beta}{\beta +1}+\frac{3\left(\beta +1\right)}{\beta}e^{\frac{3\beta +1}{\beta +1}}$ in the reconstruction problem from leaf bits. Moreover,
\[
\limsup_{n\to\infty} R^{\text{cent}} (n,q) < 1/2 \quad \text{for all} \ q \le 1/2~.
\]
\end{theorem}

\appendix

\section{Appendix}

\subsection{Preferential attachment: the moments of $N_i$, $\overline{N}_i$}
\label{sec:prefmaj}

Here we prove the analogues of Lemmas~\ref{lem:10},~\ref{lem:12},~\ref{lem:leafmoments} in the preferential attachment
model that allows us to analyze the majority rule. 

The difference with respect to uniform attachment is that, in the preferential attachment model, knowing $N_i$ at time $n-1$ is not enough to determine the probability that $N_i$ increases in the next time step. This is because the vertices counted by $N_i$ do not only have connections between them but also with other external vertices. So we introduce the \emph{weight} $w_j$, for $j\geq i$. Recall that $\wt{T}_i$ denotes the maximal size subtree of $T_{i\downarrow}^0$ with root $i$ and all other vertices unmarked. Also $N_i=|\wt{T}_i|.$ As in Section~\ref{sec:leaves}, $Y_j$ denotes the number of vertices $u\in \wt{T}_i$, such that $u\leq j$. Moreover, $\mathcal{Y}_j$ is the set of vertices $u\in \wt{T}_i$ such that $u\leq j$. Then
\begin{equation}
\label{weight} 
w_j\defeq \sum _{v\in  \mathcal{Y}_j}\left(D_v^+\left(j\right)+\beta\right)=\beta\cdot Y_{j}+\sum _{v\in \mathcal{Y}_j}D_v^+\left(j\right) ~.
\end{equation}

\noindent Similarly to Lemmas~\ref{lem:8} and~\ref{lem:9}, it is easy to see that for any positive $a,b<1$,
\begin{equation}
\label{rem:111}
e^{-1}\left(\frac{n+1-\alpha}{i+1-\alpha}\right)^b\leq\prod _{j=i}^{n-1}\left(1+\frac{b}{j+1-\alpha}\right)\leq e\left(\frac{n+1-\alpha}{i+1-\alpha}\right)^b~.
\end{equation}
Recall that in order to show the linear upper bound for the risk, we may assume that $q<1/8$ (otherwise the bound holds trivially).
\begin{lemma}
\label{lem:200}
Let $r=1-\frac{2\beta q}{\beta+1}$, $r_1=\frac{1}{\beta +1}$, and assume that $q<1/8$. Then for any $i\leq n$,
$$\frac{3\beta }{8\left(\beta +1\right)e}\left(\frac{n+1-r_1}{i+1-r_1}\right)^{r}-\frac{3\beta}{4e\left(\beta +1\right)}\leq\EXP\left[N_i\right]\leq \frac{\beta e}{1+\beta}\left(\frac{n+1-r_1}{i+1-r_1}\right)^{r}+\frac{1}{\beta +1}$$ and 
\[
\EXP\left[N_i^2\right]\leq \frac{4}{\left(1+\beta\right)^2}\left(\beta e+\beta e^2(1+\beta)+re^2(1+\beta)^2\right) \left(\frac{n+1-r_1}{i+1-r_1}\right)^{2r}~.
\]
\end{lemma}
\begin{proof}
We have 
$$\EXP\left[w_n|w_{n-1}\right]=w_{n-1}\left(1+\frac{2q+\left(1+\beta\right)\left(1-2q\right)}{n\left(\beta +1\right)-1}\right) ~,$$ since if $\mathcal{Y}_{n}$ is chosen by the new vertex $n$, then with probability $2q$ we have $w_n=w_{n-1}+1$ ($n$ is marked) and with probability $1-2q$ we have $w_n=w_{n-1}+1+\beta$ ($n$ is unmarked).
Taking expectations and expanding the resulting recurrence, we have
\begin{equation}
\label{eq:101}
\EXP\left[w_n\right]= \beta \prod _{j=i}^{n-1}\left(1+\frac{r}{j+1-r_1}\right)\leq \beta e\left(\frac{n+1-r_1}{i+1-r_1}\right)^{r}
\end{equation} 
by (\ref{rem:111}) and the fact that $w_i=\beta$. Similarly,
\begin{equation}\EXP\left[w_n\right]\geq \beta e^{-1}\left(\frac{n+1-r_1}{i+1-r_1}\right)^{r}.\label{eq:102}\end{equation}
For the second moment, we use a similar argument as in for the first moment and obtain
\begin{eqnarray*}
\EXP\left[w_n^2|w_{n-1}^2\right] &= &w_{n-1}^2+\frac{\left(1-2q\right)w_{n-1}}{\left(\beta+1\right) n-1}\left(2\left(1+\beta\right)w_{n-1}+\left(1+\beta\right)^2\right) \\
& &+\frac{2qw_{n-1}}{\left(\beta+1\right) n-1}\left(2w_{n-1}+1\right)\\
& \leq & w^2_{n-1}\left(1+\frac{2r}{n-r_1}\right)+\frac{w_{n-1}\left(\beta +1\right)r}{n-r_1}~.
\end{eqnarray*}
Taking expectations and setting $f\left(j\right)=r\left(\beta +1\right)\frac{\EXP\left[w_{j-1}\right]}{j-r_1}$, we obtain the following recurrence for $a_n\overset{\mathrm{def}}{=}\EXP\left[w_{n}\right]$:
\begin{eqnarray*}
a_n&\leq &a_{n-1}\left(1+\frac{2r}{ n-r_1}\right)+f\left(n\right)\nonumber\\
&\leq&\beta \prod _{j=i}^{n-1}\left(1+\frac{2r}{j+1-r_1}\right) +\sum _{j=i}^{n-2}f\left(j+1\right)\prod _{k=j+1}^{n-1}\left(1+\frac{2r}{k+1-r_1}\right)+f\left(n\right) \nonumber\\
& & \quad\quad \quad\quad\quad\quad \quad\quad\quad\quad\quad\quad \quad\quad\quad\quad \quad\quad \quad\quad\quad\quad \text{(since} \;\; w_i=\beta\;\text{)}\nonumber\\
&\leq &\beta e\left(\frac{n+1-r_1}{i+1-r_1}\right)^{2r}+\sum _{j=i}^{n-1}\frac{r\beta e^2\left(1+\beta\right)}{j+1-r_1}\left(\frac{j+1-r_1}{i+1-r_1}\right)^{r}\left(\frac{n+1-r_1}{j+1-r_1}\right)^{2r} \nonumber\\
& & \quad\quad \quad\quad\quad\quad \quad\quad\quad\quad\quad\quad \quad\quad\quad\quad \quad\quad \quad\quad\quad\quad\text{(by (\ref{rem:111}) and~\eqref{eq:101})}\nonumber\\
&= & \left(\frac{n+1-r_1}{i+1-r_1}\right)^{2r}\left(\beta e
+r\beta e^2\left(1+\beta\right)\left(i+1-r_1\right)^{r}\sum _{j=i}^{n-1}\left(j+1-r_1\right)^{-r-1} \right)\nonumber\\
&\leq &\left(\frac{n+1-r_1}{i+1-r_1}\right)^{2r}\left(\beta e +r\beta e^2\left(1+\beta\right)\left(i+1-r_1\right)^{r}\left(\int _{i}^{n}\left(x+1-r_1\right)^{-r-1}dx +\frac{1}{\left(i+1-r_1\right)^{r+1}}\right)\right) \nonumber\\
&\leq &\left(\beta e+\beta e^2(1+\beta)+re^2(1+\beta)^2\right) \left(\frac{n+1-r_1}{i+1-r_1}\right)^{2r}~.
\end{eqnarray*}
\noindent By \eqref{eq:101} and $Y_n=\frac{1}{1+\beta}+ \frac{w_n}{1+\beta}$, we have
\begin{eqnarray}
\EXP\left[Y_n\right] &\leq & \frac{\beta e}{1+\beta}\left(\frac{n+1-r_1}{i+1-r_1}\right)^{r}+\frac{1}{\beta +1}~.
\label{eq:103}
\end{eqnarray}
Moreover, 
\begin{eqnarray*}
\EXP\left[Y_n|Y_{n-1},w_{n-1}\right]&=&Y_{n-1}+\frac{\left(1-2q\right)w_{n-1}}{\left(\beta +1\right)\left(n-r_1\right)}.
\end{eqnarray*}
Taking expectations and expanding the resulting recurrence we obtain the following
\begin{eqnarray*}
\EXP\left[Y_n\right] &= &\frac{\left(1-2q\right)}{\beta +1}\sum _{j=i}^{n-1}\frac{\EXP\left[w_{j}\right]}{j+1-r_1}\\ & \geq& \frac{\left(1-2q\right)}{\beta +1}\sum _{j=i}^{n-1}\frac{\beta e^{-1}\left(\frac{j+1-r_1}{i+1-r_1}\right)^{r}}{j+1-r_1}\quad \text{by~\eqref{eq:102}} \\
& =&\frac{\beta \left(1-2q\right)}{e\left(\beta +1\right)\left(i+1-r_1\right)^{r}}\sum _{j=i}^{n-1}\left(j+1-r_1\right)^{r-1} \\ 
& \geq &\frac{\beta \left(1-2q\right)}{e\left(\beta +1\right)\left(i+1-r_1\right)^{r}}\int _{i}^{n-1}\left(x+1-r_1\right)^{r-1}dx \\
&\geq &\frac{3\beta }{4e\left(\beta +1\right)\left(i+1-r_1\right)^{r}}\left(\left(n-r_1\right)^{r}-\left(i+1-r_1\right)^{r}\right) \quad\text{(since $q<\frac{1}{8}$ and $\frac{1-2q}{r}\geq \frac{3}{4}$)} \\
&\geq& \frac{3\beta }{8e\left(\beta +1\right)}\left(\frac{n+1-r_1}{i+1-r_1}\right)^{r}-\frac{3\beta}{4e\left(\beta +1\right)}
\end{eqnarray*}
\noindent 
The upper bound for the second moment follows by $Y_n=\frac{1}{1+\beta}+ \frac{w_n}{1+\beta}$, hence $\mathbb{E}\left[Y_n\right]\leq\frac{4\mathbb{E}\left[w_n\right]}{\left(1+\beta\right)^2}$, and the previous computations.
\end{proof}
\noindent Denote by $\overline{Y_j}$ the number of leaf vertices in $\mathcal{Y}_j$.
\begin{lemma}\label{lem:201}
Let $r=1-\frac{2\beta q}{\beta+1}$, $r_1=\frac{1}{\beta +1}$, and assume that $q<1/8$. For any $i\leq n$,
$$\frac{\beta}{8e\left(\beta +1\right)}\left(\frac{n+1-r_1}{i+1-r_1}\right)^{r}- \frac{3\beta}{8e\left(\beta +1\right)}\leq\EXP\left[\overline{N}_i\right]\leq \frac{\beta e}{1+\beta}\left(\frac{n+1-r_1}{i+1-r_1}\right)^{r}+\frac{1}{\beta +1}$$ and 
$$\EXP\left[\overline{N}_i^2\right]\leq \frac{4}{\left(1+\beta\right)^2}\left(\beta e+\beta e^2(1+\beta)+re^2(1+\beta)^2\right) \left(\frac{n+1-r_1}{i+1-r_1}\right)^{2r}.$$
\end{lemma}
\begin{proof}
The upper bounds clearly hold by the fact that $\overline{Y}_j\leq Y_j$ and Lemma~\ref{lem:200}. Let us denote by $\overline{w}_j$ the weight of the set of leaves in $\mathcal{Y}_j$ (recall the weight function defined in~\eqref{weight}).
Notice that $\overline{w}_n=\beta \overline{Y}_n$. Hence,
\begin{eqnarray*}
\EXP\left[\overline{Y}_n|\overline{Y}_{n-1},w_{n-1},\overline{w}_{n-1}\right]&=&\overline{Y}_{n-1}+\frac{1-2q}{\left(\beta+1\right) \left(n-r_1\right)}\left(w_{n-1}-\overline{w}_{n-1}\right)\\
&=& \overline{Y}_{n-1}+\frac{1-2q}{\left(\beta+1\right) \left(n-r_1\right)}\left(w_{n-1}-\beta\overline{Y}_{n-1}\right)\\
&=& \overline{Y}_{n-1}\left(1-\frac{\beta\left(1-2q\right)}{\left(\beta +1\right)\left(n-r_1\right)}\right)+\frac{1-2q}{\left(\beta +1\right) \left(n-r_1\right)}w_{n-1}.
\end{eqnarray*}
We can assume that $i\leq n-2$, since otherwise the result can be confirmed immediately. Let $f\left(n\right)=\frac{1-2q}{\left(\beta +1\right)\left(n-r_1\right)}\EXP\left[w_{n-1}\right]$. Then, $a_n\overset{\mathrm{def}}{=}\EXP\left[\overline{Y}_n\right]$ satisfies 
\begin{eqnarray}
a_n&=&a_{n-1}\left(1-\frac{\beta \left(1-2q\right)}{\left(\beta +1\right)\left( n-r_1\right)}\right)+f\left(n\right)\nonumber\\
&\geq &\sum _{j=i}^{n-2}f\left(j+1\right)\prod _{k=j+1}^{n-1}\left(1-\frac{\beta\left(1-2q\right)}{\left(\beta +1\right)\left(k+1-r_1\right)}\right)\nonumber\\
&\geq &\sum _{j=i}^{n-2}\frac{\beta\left(1-2q\right)}{e\left(\beta +1\right)\left(j+1-r_1\right)} \left(\frac{j+1-r_1}{i+1-r_1}\right)^{r}\frac{j+1-r_1}{n+1-r_1}\quad\quad\text{(by \eqref{eq:102})}\nonumber\\
&\geq &\frac{\beta\left(1-2q\right)}{e\left(\beta +1\right)\left(n+1-r_1\right)}\left(i+1-r_1\right)^{-r}\int _{i}^{n-2}\left(x+1-r_1\right)^{r}dx\nonumber\\
&\geq &\frac{3\beta}{8e\left(\beta +1\right)}\left(\frac{1}{3}\left(\frac{n+1-r_1}{i+1-r_1}\right)^{r}-1\right) ~.\nonumber
\end{eqnarray}
\end{proof}

\subsection{Proof of Theorem \ref{thm:centroid2}}
\label{sec:prefcent}

To show the theorem, we work as in Section~\ref{sec:centroid}. For
brevity, we omit overlapping arguments and we only fill in the missing
points. Recall that the estimator $\wh{b}_\text{cent}$ is the bit
value of the centroid $v^*$ of the tree. In case there are two
centroids we pick one uniformly at random. However, the probability of
this event tends to zero, see Wagner and Durant \cite[Lemma 15]{wagner2019centroid}. 

\begin{theorem}\label{lem:centrdist2}\upshape{(Wagner and Durant \cite[Theorem 9, Theorem 11]{wagner2019centroid})}
Let $\delta_n$ be the depth of the centroid closest to the root and $L_n$ be its label at time $n$. Then
$$\lim_{n\rightarrow\infty} \mathbb{E}\left[\delta_n\right]=\frac{\beta}{\beta +1}\;\;\text{and}\;\;\lim_{n\rightarrow\infty}\PROB\left\{
  L_n=0\right\}=1-\beta \left(2^{1/(1+\beta)}-1\right)~.$$
\end{theorem}

We may combine the above theorem and equation~\eqref{eq:1002} as follows.
\begin{eqnarray*}
   \limsup_{n\to \infty} \PROB\left\{ \wh{b}_\text{cent}\neq B_0
    \right\} &\le& \frac{1}{2} -\frac{1}{2} \liminf_{n\to \infty} \PROB\{ D=0 \}\\
& = & \frac{1}{2} -\frac{1}{2} \liminf_{n\to \infty} \PROB\{ \delta _n=0 \} \\
             &= &  \frac{1}{2} -\frac{1}{2} \left(1-\beta
                  \left(2^{1/(1+\beta)}-1\right) \right)
                  <\frac{1}{2}~.
\end{eqnarray*}

The rest follows directly by combining Theorem~\ref{lem:centrdist2} and equation~\eqref{eq:502}.

\noindent To show Theorem~\ref{thm:centroid2} in the case of reconstruction from leaf-bits, we prove the following lemma.
\begin{lemma}
$\PROB\left\{\Delta >2 \right\}\leq \frac{3}{\beta}\left(\beta +1\right)e^{\frac{3\beta +1}{\beta +1}} n^{-\frac{1}{\beta +1}}+\mathcal{O}\left(\frac{1}{n}\right)~.
$\label{delta5}
\end{lemma}
\begin{proof}
Denote by $N_1$ the set of vertices $i\le \lceil n/2\rceil$ at distance one from the root. For vertex $u$ such that $\lceil n/2\rceil<u\leq n$, we write $Y_u$ for the indicator that $u$ attaches to a vertex in $N_1$ (say it attaches to $u_1$) and also an independent $\text{Bernoulli}\left(\frac{D^+_{u_1}\left(\lceil n/2\rceil\right)}{D^+_{u_1}\left(u-1\right)}\right)$ coin flip is successful. We add the last condition so that 
\[
\PROB\{Y_uY_v=1\}=\PROB\{Y_u=1\}\PROB\{Y_v=1\}~,
\]
for any $u,v$ such that $v>u>\lceil n/2\rceil$. We write $X_u$ for the indicator that $u$ is connected with an edge to $N_1$ and is a
leaf. Then, $X_u=Y_uZ_u$, where $Z_u$ is the indicator that no
vertex $t>u$ attaches to $u$.
Moreover,
\[
\PROB\{Z_v=1|Y_uY_v=1\}=\PROB\{Z_v=1|Y_v=1\}
\]
when $v>u$, and
\begin{equation*}
\PROB\{ X_uX_v=1\}
 =  \PROB\{ Z_u=1|Z_vY_uY_v=1 \} \PROB\{ Z_v=1|Y_uY_v=1\} \PROB\{ Y_uY_v=1\}~.
\end{equation*}
Combining the previous observations, we obtain for $v>u$:
\begin{eqnarray*}
\lefteqn{
\cov \left(X_uX_v\right)    } \\
& = & \PROB\{ Z_v=1|Y_uY_v=1\} \PROB\{ Y_uY_v=1\} \left(  \PROB\{ Z_u=1|Z_vY_uY_v=1 \}- \PROB\{ Z_u=1|Y_u=1 \} \right)~.
\end{eqnarray*}
But 
\begin{eqnarray*}
\PROB\left\{ Z_u=1|Y_uY_vZ_v=1  \right\}&=&\frac{u}{u+1-\frac{1}{\beta+1}}\cdots \frac{v-2}{v-1-\frac{1}{\beta+1}}\cdot\frac{v-\frac{\beta}{\beta +1}}{v}\cdots \frac{n-1-\frac{\beta}{\beta +1}}{n-1}\\
&\leq &\frac{u}{u+\frac{\beta}{\beta +1}}\cdots \frac{v-2}{v-2+\frac{\beta}{\beta +1}}\cdot\frac{v}{v+\frac{\beta}{\beta +1}}\cdots \frac{n-1}{n-1+\frac{\beta}{\beta +1}}
\end{eqnarray*}
\noindent and
\begin{eqnarray*}
\PROB\left\{ Z_u=1|Y_u=1  \right\}&=& \frac{u}{u+1-\frac{1}{\beta+1}}\cdots\frac{n-1}{n-\frac{1}{\beta +1}}\;=\; \frac{u}{u+\frac{\beta}{\beta+1}}\cdots\frac{n-1}{n-1+\frac{\beta}{\beta +1}}~.
\end{eqnarray*}
Therefore, for $w\left(N_1\right)=\sum _{i\in N_1}\left(D^+_i\left(\lceil n/2\rceil\right)+\beta\right)$, we have
\begin{eqnarray*}
\cov \left(X_uX_v\right) &\leq & \left(1-\frac{v-1}{v-1+\frac{\beta}{\beta +1}}\right)\cdot \EXP\left\{\frac{w\left(N_1\right)}{\left(\beta+1\right)u-1}\right\}^2\\
&\leq& \frac{2}{n}\cdot \EXP\left\{\frac{w\left(N_1\right)}{\left(\beta+1\right)u-1}\right\}^2\\
& \leq & \frac{8}{n^3\left(\beta +1\right)^2}\cdot \EXP\left\{w\left(N_1\right)\right\}^2 ~,
\end{eqnarray*}
since $v>u\geq n/2+1$.
Moreover,
\begin{eqnarray*}\EXP X_u &= &\left(\frac{u}{u+\frac{\beta}{\beta+1}}\cdots\frac{n-1}{n-1+\frac{\beta}{\beta +1}}\right)\cdot \EXP\left\{\frac{w\left(N_1\right)}{\left(\beta+1\right)u-1}\right\}\\
& \geq & e^{-\frac{\beta}{\beta +1}}\cdot
\EXP\left\{\frac{w\left(N_1\right)}{\left(\beta+1\right)n}\right\}~.
\end{eqnarray*}
Then, by Chebyshev's inequality and the previous bounds,
\begin{eqnarray*}
\PROB\left\{ \sum_{i > \lceil n/2\rceil } X_i =0 \right\} 
& \leq & \frac{\sum \limits_{\substack{i\geq
  \lceil n/2\rceil}}\var(X_i)+\sum \limits_{\substack{i\neq j\\ i\geq
  \lceil n/2\rceil}}\cov(X_iX_j) }{\left(  \sum \limits_{i \ge \lceil n/2\rceil} \EXP X_i \right)^2}\\
  & \leq & \frac{e^{\frac{2\beta}{\beta +1}}\left(\beta +1\right)}{\EXP\left\{w\left(N_1\right)\right\}}+\mathcal{O}\left(\frac{1}{n}\right)~.
\end{eqnarray*}
Moreover $\EXP\left\{w\left(N_1\right)\right\}\geq \frac{\beta}{3e}n^{\frac{1}{\beta +1}}$. To see that, notice that a its expectation satisfies the recurrence
 $$\alpha_n\;\geq\;\alpha_{n-1}\left(1+\frac{1/\left(\beta +1\right)}{n-1/\left(\beta +1\right)}\right)~,$$ with initial condition $\alpha_1=\beta $, and then we can apply \eqref{rem:111}).
\end{proof}
\noindent By Lemma~\ref{delta5} and~\cite[Theorem 6.50]{drmota2009random},
\begin{eqnarray*}\EXP \Delta & =&\sum _{i=0}^{n-1}\PROB\left\{\Delta >i \right\}\leq 2+\frac{3}{\beta}\left(\beta +1\right)e^{\frac{3\beta +1}{\beta +1}}   +\sum _{i>n^{1/(\beta +1)}}\PROB\left\{\Delta >i \right\}+o_n\left(1\right)\\
&=&2+\frac{3}{\beta}\left(\beta +1\right)e^{\frac{3\beta +1}{\beta +1}} +o_n\left(1\right).\end{eqnarray*} As in Section~\ref{sec:centroidleaves} and using Theorem~\ref{lem:centrdist2}, Lemma~\ref{delta5},
 we have that, if $\wt{v}$ is a leaf vertex that is closest to the centroid $v^*$,
then
\[
  \limsup_{n\to \infty} \EXP d(\wt{v},0)\leq   \EXP [\Delta + 2D] \leq 2+\frac{2\beta}{\beta +1}+\frac{3}{\beta}\left(\beta +1\right)e^{\frac{3\beta +1}{\beta +1}}  .
\]
This completes the proof of the first part of
Theorem~\ref{thm:centroid2} for the reconstruction problem from leaf
bits. 
The second part follows from the fact that the root is the centroid of
the tree with probability bounded away from zero, combined with the fact
that the expected distance of the nearest leaf is bounded, as shown above.



\end{document}